\documentclass[11pt,twoside,a4paper]{article}
\usepackage[english]{babel}
\usepackage{a4wide}
\usepackage{graphicx}
\usepackage{amsmath,amssymb,amsthm, amsgen}
\usepackage{listings}
\usepackage{color}
\usepackage[usenames,dvipsnames]{xcolor}
\usepackage{rotating}
\usepackage{hhline}
\usepackage{multirow}
\usepackage{enumerate}
\usepackage[bookmarks]{}
\usepackage{amsfonts}   % if you want the fonts
\usepackage{dsfont}
\usepackage{tikz}
\usepackage{booktabs}
\usepackage{enumitem}
\usepackage{url}

\newtheorem{thm}{Theorem}[section]
\newtheorem{prop}[thm]{Proposition}
\newtheorem{lem}[thm]{Lemma}

\newtheorem{cor}[thm]{Corollary}

\newtheorem{assn}[thm]{Assumption}
\newcommand{\Eint}{E^{\operatorname{int}}}
\newcommand{\hEnint}{\hat E_n^{\operatorname{int}}}

\newcommand{\id}{\operatorname{id}}

\newcommand{\sign}{\operatorname{sign}}
\newcommand{\supp}{\operatorname{supp}}

\newcommand{\weakto}{ \rightharpoonup }
\newcommand{\xto}[1]{\xrightarrow{#1}}

\newcommand{\C}{\mathbb C}
\newcommand{\e}{\varepsilon}
\newcommand{\N}{\mathbb N}
\newcommand{\R}{\mathbb R}

\newcommand{\ba}{\mathbf a}
\newcommand{\bb}{\mathbf b}

\newcommand{\bx}{\mathbf x}

\newcommand{\cM}{\mathcal M}

\newcommand{\cP}{\mathcal P}

\newcommand{\ov}[1]{\overline{#1}}
\newcommand{\oba}{\overline \ba}
\newcommand{\obx}{\overline \bx}

\newcommand{\omu}{\overline \mu}

\newcommand{\oR}{\overline R}
\newcommand{\orho}{\overline \rho}

\newcommand{\ox}{\overline x}

\DeclareFontFamily{U}{mathx}{\hyphenchar\font45}
\DeclareFontShape{U}{mathx}{m}{n}{
      <5> <6> <7> <8> <9> <10>
      <10.95> <12> <14.4> <17.28> <20.74> <24.88>
      mathx10
      }{}
\DeclareSymbolFont{mathx}{U}{mathx}{m}{n}
\DeclareFontSubstitution{U}{mathx}{m}{n}
\DeclareMathAccent{\widecheck}{0}{mathx}{"71}
\def\showComm{0} % 0 = comments off,  1 = comments on
\long\def\comm#1{{\if\showComm1 {\color{magenta} #1 } \fi}}

\title{Convergence rates for energies of interacting particles whose distribution spreads out as their number increases}
\author{Patrick van Meurs and Ken'ichiro Tanaka}

\begin{document}

\maketitle

\begin{abstract}
We consider a class of particle systems which appear in various applications such as approximation theory, plasticity, potential theory and space-filling designs. The positions of the particles on the real line are described as a global minimum of an interaction energy, which consists of a nonlocal, repulsive interaction part and a confining part. Motivated by the applications, we cover non-standard scenarios in which the confining potential weakens as the number of particles increases. This results in a large area over which the particles spread out. Our aim is to approximate the particle interaction energy by a corresponding continuum interacting energy. Our main results are bounds on the corresponding energy difference and on the difference between the related potential values. We demonstrate that these bounds are useful to problems in approximation theory and plasticity. The proof of these bounds relies on convexity assumptions on the interaction and confining potentials. It combines recent advances in the literature with a new upper bound on the minimizer of the continuum interaction energy.
\end{abstract}


% MSC 2020

%74Q05, % Homogenization in equilibrium problems of solid mechanics
%82B21 % continuum models 
%82D35: application to metals
%35A15: PDE: variational methods
%74G10 % Analytic approximation of solutions (perturbation methods, asymptotic methods, series, etc.)
%82C21, %Dynamic continuum models (systems of particles, etc.)

\noindent
\hrulefill

\tableofcontents

%---------------------------------
%\bigskip
%\noindent
%\hrulefill
%
%\kt{Command for comments by Kenichiro}
%
%\pvm{Command for comments by Patrick}
%
%{\color{magenta} Comments for additional computations / explanation. These can be turned on or off by changing the boolean "showComm" in the tex above "begin document" at the very end of the preamble.} 

%\section*{Work plan}
%\label{s:goal}

%From 18 Nov and onwards:
%\begin{itemize}
%  \item PvM fixes the comments of the Zoom call on Wed 17 Nov
%  \item PvM makes a numerics section
%  \item KT updates and revises Sections 1 and 2
%\end{itemize}

\noindent
\hrulefill
%---------------------------------

%---------------------------------
\section{Introduction}
\label{s:intro}

This paper studies the behaviour of particle systems in which the unknowns are the particle positions $\ba := (a_1, \ldots, a_n) \in \R^n$ on the real line with $a_1 < \ldots < a_n$. The starting point is the minimization of the particle interaction energy given by 
\begin{equation} \label{In}
  I_n^D (\ba) 
  := \frac n{n-1} \sum_{i =1}^n \sum_{ \substack{ j = 1 \\ j \neq i }}^n K (a_i - a_j) + 2 \beta \sum_{i=1}^n Q ( a_i ), 
\end{equation}
where the superscript `$D$' stands for `discrete' (i.e., a finite number of particles),
\begin{equation}
  K(x) := -\log| \tanh x |
  \label{eq:def_K}
\end{equation}
is the repulsive particle interaction potential, $Q$ is a confining potential, and $\beta = \beta_n > 0$ is a given parameter which controls the strength of the confinement. Figure \ref{fig:V} illustrates $K$ and a typical $Q$. 
Assumption \ref{a:Q} lists our assumptions on $Q$.

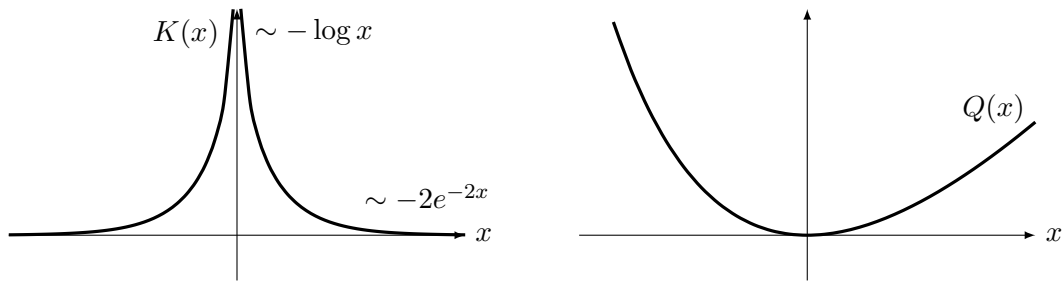
\begin{figure}[h]
\centering
\begin{tikzpicture}[scale=1.5, >= latex]    
\def \w {2}
\def \v {3}

\begin{scope}[scale=.667]        
\draw[->] (0,-.6) -- (0,\v);
\draw[->] (-\v,0) -- (\v,0) node[right] {$x$};
\draw[domain=0.05:\v, smooth, very thick] plot (\x,{-ln(tanh \x)});
\draw[domain=-\v:-0.05, smooth, very thick] plot (\x,{{-ln(-tanh \x)}});
\draw (.1, \v) node[anchor = north west]{$\sim - \log x$};
\draw (2.5, .5) node {$\sim - 2 e^{-2x}$};
\draw (-.1, \v) node[anchor = north east]{$K(x)$};
\end{scope} 

\begin{scope}[shift={(2.5*\w,0)},scale=1]
\draw[->] (0,-.4) -- (0,\w);
\draw[->] (-\w,0) -- (\w,0) node[right] {$x$};
\draw[domain=-1.7:\w, smooth, very thick] plot (\x,{ 6*(4/(\x + 4) + \x/4 -1) });
\draw (\w, .9) node[anchor = south east]{$Q(x)$};
\end{scope} 
\end{tikzpicture} \\
\caption{Sketches of $K$ and $Q$.}
\label{fig:V}
\end{figure} 

\begin{assn} \label{a:Q}
The confinement potential $Q$ satisfies:
\begin{enumerate}[label=(\roman*)]
  \item (Regularity) $Q \in C^4(\R)$,
  %\item \label{a:Q:min0} (Bound from below) $\inf_\R Q > -\infty$,
  \item \label{a:Q:cv} (Convexity) $Q$ is convex on $\R$,
  \item (Growth) $\displaystyle \liminf_{x \to \pm \infty} Q(x) > \inf_\R Q$.
\end{enumerate} 
\end{assn}

Energies of the type \eqref{In} appear in various applications such as function approximation theory \cite{TanakaSugihara19,HayakawaTanaka19}, plasticity of metals \cite{GeersPeerlingsPeletierScardia13,
GarroniVanMeursPeletierScardia16,
VanMeurs21DBLArXiv}, potential theory \cite{SandierSerfaty15} and space-filling designs \cite{PronzatoZhigljavsky20}.
In this paper we focus on the first two. 
First, for the application to approximation theory in \cite{TanakaSugihara19,HayakawaTanaka19}, 
the authors consider the energy~\eqref{In} with $\beta = 1$
for minimizing the worst case error of approximation formulas given by
\begin{align}
E_{n}^{\mathrm{min}} := \inf \left[
\sup_{\| f \| \leq 1} \sup_{x \in \R}
\left|
f(x) - \sum_{j=1}^{l} \sum_{k=0}^{m_{j}-1} f^{(k)}(a_{j}) \, \phi_{jk}(x)
\right|
\right],    
\label{eq:WCE}
\end{align}
where the infimum is taken over all possible approximation formulas given by the sampling points $a_{j}$ and the functions $\phi_{jk}$, and the first of the two suprema is taken over a certain function space with a weight on $\R$. 
Since the value $E_{n}^{\mathrm{min}}$ is given by a certain potential function of 
the sampling points $a_{j}$, 
their (nearly) optimal locations can be determined by 
minimizing the energy~\eqref{In} with $Q$ defined by the weight%
\footnote{In \cite{TanakaSugihara19,HayakawaTanaka19}, 
the authors consider the interaction potential
$K(y) = -\log| \tanh (C y) |$
with a parameter $C > 0$. 
In this paper, 
we rescale $x = Cy$ to remove this parameter without loss of generality.
}.  
Second, in the application to plasticity of metals, $a_i$ are the locations of microscopic defects in the material. 
In this case, an interaction potential different from $K$ is used, but most of the properties are conserved 
(such as evenness, the logarithmic singularity, the exponentially decaying tails, and the convexity on $(0,\infty)$).
Furthermore, $\beta > 0$ regulates the strength of the external force, which may depend on $n$ without any a priori upper or lower bound in terms of $n$. 

One of the focal points in this literature -- in particular in the two applications we focus on -- is to construct sharp estimates on the difference between $I_n^D$ evaluated at its minimal point $\oba$ (or the value of the corresponding potential at $\oba$; see \eqref{FD:FC}) and its continuum counterpart. The continuum energy corresponding to $I_n^D$ obtained in \cite{TanakaSugihara19} is
\begin{align} \label{InC}
  I_n^C : n \cP(\R) \to \R \cup \{\infty\}, \qquad  
  I_n^C (\mu) 
  &:= \int_\R \int_\R K(x-y) \, d \mu (y) d \mu (x) + 2 \beta \int_\R Q(x) \, d \mu(x), 
\end{align}
where $\cP(\R)$ is the space of probability measures on $\R$, and $n \cP(\R) = \{ n \mu \mid \mu \in \cP(\R) \}$. From the viewpoint of the application to plasticity -- in particular in \cite{KimuraVanMeurs21,VanMeurs21DBLArXiv} -- it is desired to construct sufficiently sharp lower and upper bounds on 
\begin{equation} \label{ene:diff} 
  I_n^D(\overline{\mathbf{a}}) - I_n^C (\overline{\mu}), 
\end{equation}
where $\omu$ is the minimizer of $I_n^C$. Such bounds were recently obtained in \cite{KimuraVanMeurs21}, but only for the specific scaling $\beta \sim n$. In this paper we extend these bounds to a wide scaling regime for $\beta$ (see Theorem \ref{t:intro}), which in particular covers the setting in approximation theory where $\beta = 1$. 

From the viewpoint of  the application to approximation theory, it is desired to find sharp lower and upper bounds on 
\begin{equation} \label{pot:diff} 
  F_n^D - F_n^C
\end{equation}
rather than bounds on \eqref{ene:diff}. Here,  
\begin{equation} \label{FD:FC}
  F_n^D := I_n^D(\overline{\mathbf{a}}) - \beta \sum_{i=1}^n Q(\overline{a}_i),
  \qquad
  F_n^C := I_n^C (\overline{\mu}) - \beta \int_\R Q(x) \, d\overline{\mu}(x)
\end{equation}
are the potential values corresponding to the energies $I_n^D$ and $I_n^C$ at their minimizers. The values $F_n^D$ and $F_n^C$ are commonly called the Robin constants; see e.g.\ \cite{SaffTotik97}. These constants provide upper and lower bounds on the error value defined in \eqref{eq:WCE}. Indeed, \cite[Theorem 3.5]{TanakaSugihara19} states that
\begin{equation} \label{cEn:bds} 
  - \frac{ F_n^C }{n-1}
  \leq \log E_{n}^{\mathrm{min}}
  \leq - \frac{ F_n^D }n.
\end{equation}
Recently, in \cite[Theorem 2.3]{HayakawaTanaka19}, upper and lower bounds on the difference of the Robin constants
\eqref{pot:diff} were derived: 
\begin{equation} \label{HT19:estimate:I} % p.49 
   - \frac{n+1}{n-1} F_n^D - (3 + \log 2) \frac{n^2}{n-1}
   \leq F_n^D - F_n^C \leq 0.
\end{equation}
While the upper bound is satisfactory, the lower bound is suboptimal; the least one expects is $|F_n^D - F_n^C| \ll F_n^D$ as $n \to \infty$.  
In this paper, 
our second goal is to improve the lower bound in \eqref{HT19:estimate:I}; see Theorem \ref{t:intro}.
In addition, Theorem \ref{t:intro} extends the class of potentials $Q$ under which \eqref{HT19:estimate:I} was derived.

The rest of this paper is organized as follows. 
In Section~\ref{s:main} we present our main theorems: 
Theorems~\ref{t:intro} and~\ref{t}. 
Theorem \ref{t} is a rescaled version of Theorem \ref{t:intro} which fits more naturally to the application to plasticity. 
We also state simplified versions of these theorems (see Corollaries~\ref{c:t} and~\ref{c:t:intro}) and discuss applications thereof. 
In Section~\ref{s:pf:t} we prove Theorem~\ref{t}. This proof is the main mathematical contribution of this paper.
In Section~\ref{s:pf:t:intro} we show that Theorem~\ref{t:intro} follows from Theorem~\ref{t}.
%Finally, in Section~\ref{s:num} we show that our results improve the bounds in~\eqref{HT19:estimate:I} in the application to function approximation. 

%---------------------------------
\section{Main results}
\label{s:main} 

\subsection{Preliminaries}
\label{s:main:prel}

Let $K$ be given by \eqref{eq:def_K} and 
$Q$ be a potential satisfying Assumption~\ref{a:Q}.
Several useful consequences of Assumption \ref{a:Q} are that $Q$ is bounded from below, that the minimum of $Q$ over $\R$ 
is attained, and that $Q$ has at least linear growth, i.e.\ there exists $c, C > 0$ such that
\begin{equation} \label{Q:lin:growth}
  Q(x) \geq c|x| - C
  \qquad \text{for all } x \in \R.
\end{equation}
The growth of $Q$ in \eqref{Q:lin:growth} guarantees the existence of minimizers $\oba$ and $\omu$ of $I_n^D$ and $I_n^C$ respectively. Since we don't assume strict convexity of $Q$, $\oba$ may not be unique, but $\omu$ is \cite[Theorem 1.5]{KimuraVanMeurs20DOI}. We fix $\oba$ as one of the minimizers. \comm{Example of nonuniqueness of minimizers: $n = 2$, $Q(x) = |x|$ (can be regularized around $0$ to meet the $C^2$ requirement). The obvious minimizer looks like $a_2 = -a_1 = c$ for some $c > 0$. But, for any $|x| \leq c$, $b_1 = a_1 + x$ and $b_2 = a_2 + x$ satisfies $I_n^D(\bb) = I_n^D(\ba)$.}

To state our main result in a concise manner, we make further assumptions on $Q$ without loss of generality. Since the interaction terms in $I_n^D$ and $I_n^C$ are invariant to spatial translation of the particles, we may assume that $Q(x)$ is minimal at $x=0$. Moreover, adding a constant $C$ to $Q$ is invariant to either $I_n^D - I_n^C$ and $F_n^D - F_n^C$, and thus we may assume $Q(0) = 0$. In conclusion, we may assume without loss of generality that
\begin{equation} \label{Q:min0}
  \min_\R Q = Q(0) = 0.
\end{equation} 
For later use, we set
\begin{equation} \label{J:q1:q2}
  [q_1, q_2] := \{ x \in \R \mid Q(x) = 0 \}
\end{equation}
as the set of all minimizers of $Q$. Here, we allow for the generic case $q_1 = 0 = q_2$, for which we define $[q_1, q_2] := \{0\}$.

The last additional assumption which we make on $Q$ without loss of generality is more technical. Let 
\[
  P(x) := \int_0^x Q(y) \, dy \qquad \text{for all } x \in \R
\]
be a primitive of $Q$, and let $P^{-1} : \R \setminus \{0\} \to \R$ be the inverse of $P |_{\R \setminus [q_1, q_2] }$. After fixing $n \geq 1$ and $\beta = \beta_n > 0$ we assume that
\begin{equation} \label{Pinv:ass}  
  P^{-1} \Big( \frac n\beta \Big) 
  \geq - P^{-1} \Big( -\frac n\beta \Big).
\end{equation}
If \eqref{Pinv:ass} would not hold, then by the change of spatial variable from $x$ to $-x$ the resulting energies are given by \eqref{In} and \eqref{InC} with $Q(x)$ replaced by $Q(-x)$, for which \eqref{Pinv:ass} holds. In the remainder, we will always assume that \eqref{Q:min0} and \eqref{Pinv:ass} hold in addition to Assumption \ref{a:Q} unless mentioned otherwise. 

Since $P$ will play an important role in the statements of Theorems \ref{t:intro} and \ref{t}, we list several of its properties: 
\begin{itemize}
  \item $P$ is strictly increasing on $\R \setminus (q_1, q_2)$ and $P |_{[q_1, q_2]} \equiv 0$,
  \item $P$ is concave on $(-\infty,0]$ and convex on $[0,\infty)$,
  \item there exist $c,C > 0$ such that $|P(x)| \geq c x^2 - C$ for all $x \in \R$,
  \item $P^{-1}$ is increasing, 
  \item $P^{-1}$ is convex on $(-\infty,0)$ and concave on $(0,\infty)$,
  \item $P^{-1}(y) \to q_1 \leq 0$ as $y \nearrow 0$ and
        $P^{-1}(y) \to q_2 \geq 0$ as $y \searrow 0$,
  \item $\displaystyle \lim_{|y| \to \infty} |P^{-1}(y)| = \infty \quad$ and $\quad \displaystyle  \limsup_{|y| \to \infty} \frac{|P^{-1}(y)|}{\sqrt{|y|}} < \infty$.
\end{itemize}

Finally, we provide two typical examples of $Q$ and $P$. The first is $Q(x) = |x|^p$ for $p \geq 1$. It satisfies Assumption \ref{a:Q} (except for $1 \leq p < 2$, in which case the regularity requirement fails, but that is irrelevant in this example) and
\begin{equation} \label{P:example}
  P(x) = \frac{\sign (x)}{p+1} |x|^{p+1},
\qquad P^{-1}(y) = \sign(y) \big((p+1) |y|\big)^{\tfrac1{p+1}}.
\end{equation}
The second example demonstrates the situation where $q_1 < q_2$. Let $Q(x) = [|x|-1]_+$. It satisfies Assumption \ref{a:Q} (except, again, for the regularity requirement) and
\[
  P(x) = \begin{cases}
    -(x+1)^2
    &\text{if } x < -1 \\
    0
    &\text{if } -1 \leq x \leq 1 \\
    (x-1)^2
    &\text{if } x > 1, 
  \end{cases}
  \qquad %\text{and} \quad
  P^{-1}(y) = \begin{cases}
    - \sqrt{|y|} - 1
    &\text{if } y < 0 \\
    \sqrt y + 1
    &\text{if } y > 0.
  \end{cases}
\]

\subsection{Main Theorems}
\label{s:main:thms}

With the preliminaries in Section \ref{s:main:prel} we are ready to state the first of our two main theorems on upper and lower bounds of \eqref{ene:diff} and \eqref{pot:diff}:

\begin{thm} \label{t:intro}
Let $Q$ satisfy Assumption \ref{a:Q}, \eqref{Q:min0} and \eqref{Pinv:ass}, and let $q_2$ be as in \eqref{J:q1:q2}. Then for all $\Gamma > 0$ there exist $C, C' > 0$ such that for all $n \in \N$ with $n > \max \{q_2, 1\}$ and all $\frac n{P(n)} \leq \beta \leq \Gamma n$ there holds $\supp \omu \subset [-C \alpha, C \alpha]$, \comm{[$n=1$ is not allowed because of the factor $\frac1{n-1}$ in $I_n^D$]}
\[
  - \sqrt{ \frac{n^2}\alpha B_n^\beta }
   \leq F_n^D - F_n^C  
   \leq 0
\]
and 
\[
  - B_n^\beta
   \leq I_n^D(\oba) - I_n^C (\omu) 
   \leq B_n^\beta,
\] 
where
\[
  B_n^\beta := C' \frac{n^2}\alpha \min \Big\{ \frac \alpha n \log \Big( \frac n\alpha \Big( 1 + \frac{\beta \alpha^3}n \| Q'' \|_{L^\infty(\supp \omu)} \Big) \Big), 1 \Big\},
  \qquad \alpha := P^{-1} \Big( \frac n\beta \Big).
\]
\end{thm} 

Out of the four bounds, the upper bound on $F_n^D - F_n^C$ is not new:

\begin{thm}[$\sim${\cite[Theorems 3.4 and 3.5]{TanakaSugihara19}}] \label{t:TS19} 
Under the conditions in Theorem \ref{t:intro}, there holds
\[
  F_n^D - F_n^C  
   \leq 0.
\]
\end{thm} 

Theorem \ref{t:TS19} is not a \emph{direct} citation of \cite[Theorems 3.4 and 3.5]{TanakaSugihara19}; we give further details. First, for the reader's convenience we note that \cite{TanakaSugihara19} uses the potentials 
\[
  F_{K,Q}^{\rm C}(n) := F_n^C
  \quad \text{and}
  \quad F_{K,Q}^{\rm D}(n) := \frac{n-1}n F_n^D.  
\]
\cite[Theorem 3.5]{TanakaSugihara19} states that $F_{K,Q}^{\rm D}(n) - F_{K,Q}^{\rm C}(n) \leq 0$. In addition, from the statement of \cite[Theorem 3.4]{TanakaSugihara19} and the proof of \cite[Theorem 3.5]{TanakaSugihara19} it is easy to see that the stronger estimate $\frac n{n-1} F_{K,Q}^{\rm D}(n) - F_{K,Q}^{\rm C}(n) \leq 0$ also holds. This estimate is equivalent to that in Theorem \ref{t:TS19} above. Second, in the setting of \cite{TanakaSugihara19}, $\beta = 1$ and $Q$ is analytic and strictly convex. Nevertheless, the proofs of \cite[Theorems 3.4 and 3.5]{TanakaSugihara19} apply without modification to the (weaker) assumptions on $\beta$ and $Q$ used in the present paper. Indeed, the proofs rely on potential theory and the fact that $K$ is the restriction to $\R$ of an harmonic map on a strip in $\C$. \comm{These assumptions are used elsewhere; namely analyticity is used to derive the integral form of the error of the approximation formula (i.e.\ $\| f_n - f \|$ in terms of a contour integral in the complex plane), and strong convexity is used to get uniqueness of the minimizer of $I^{\rm D}$.}
\smallskip

Before discussing the other three bounds in Theorem \ref{t:intro} and their proofs, we reformulate Theorem \ref{t:intro} into a form (see Theorem \ref{t} below) in which several quantities scale as $1$ as $n \to \infty$. Indeed, in the current setting, the minimizer $\omu$ spreads out over a length scale $\alpha$ as $n$ gets large, and it is unclear how the energy values $I_n^D(\oba)$ and $I_n^C (\omu)$ (and therefore also the potential values $F_n^D$ and $F_n^C$) scale as $n$ gets large. In the following we rescale these quantities to order $1$ and motivate the expression for $\alpha$.

We motivate the expression for $\alpha$ in a formal manner. We want to redefine $\alpha = \alpha (n, \beta)$ as the length scale over which the measure $\omu$ spreads out as $n \to \infty$. Then, to any measure $\mu \in n \cP(\R)$ we associate the measure
\[
  \rho \in \cP(\R), \qquad 
  \rho = \frac1n \Big( \frac1\alpha \id \Big)_\# \mu,
\]
where `$\#$' in the subscript denotes the push-forward. 
By construction, $\rho$ has mass $1$. Moreover, if the mass of $\mu$ spreads out over a length scale of order $\alpha$, then the mass of $\rho$ spreads out over a length scale of order $1$. Using this rescaling of measures, we introduce the rescaled continuum energy 
\begin{equation} \label{E:def}   
  E^\alpha : \cP(\R) \to \R \cup \{\infty\}, \qquad 
  E^\alpha(\rho) := \frac1\gamma I_n^C(n (\alpha \id)_\# \rho),
\end{equation}
where we want the parameter $\gamma = \gamma(n, \beta) > 0$ to be proportional to $I_n^C(\omu)$ as $n \to \infty$. Then, since $\omu$ is the unique minimizer of $I_n^C$, it follows that
\[
  \orho := \frac1n \Big( \frac1\alpha \id \Big)_\# \omu
\]
is the unique minimizer of $E^\alpha$. Note that $\orho$ spreads out over a length scale of order $1$ and that $E^\alpha(\orho)$ is of order $1$.

Next we derive explicit expressions for $\alpha$ and $\gamma$ with the aforementioned properties. Since we do not have an explicit expression for $\omu$, this is not straightforward. Instead, we apply the method performed in \cite{ScardiaPeerlingsPeletierGeers14}; we choose a certain $\rho \in \cP(\R)$ for which we expect $E^\alpha(\rho)$ to be of the same order as $E^\alpha(\orho)$, and for which we can compute explicitly the interaction and confinement part $E^\alpha(\rho)$. Then, we take $\gamma$ and $\alpha$ such that both the interaction and confinement part are of order $1$. A posteriori we will verify that then also for $E^\alpha(\orho)$ the interaction and confinement parts are of order $1$. 

We take $\rho$ to be absolutely continuous with density $\chi_{[0,1]}$, where
\begin{equation} \label{chi} 
  \chi_A(x) := \left\{ \begin{array}{ll}
    1
    & \text{if } x \in A \\
    0
    &\text{otherwise}
  \end{array} \right.
\end{equation} 
is the indicator function of some $A \subset \R$. From \eqref{E:def} we obtain that the interaction part of $E^\alpha(\rho)$ is
\begin{equation} \label{pf:zk}
  \frac{n^2}{\gamma} \int_\R \int_\R K( \alpha[x-y]) \, d\rho (y) d\rho (x) 
  = \frac{n^2}{\gamma \alpha} \int_0^1 \int_0^1 \alpha K( \alpha[x-y]) \, dydx.
\end{equation}
The function
\begin{equation} \label{Kal}
    K_\alpha(x) := \alpha K (\alpha x)
\end{equation}
in the integrand satisfies $\int_\R K_\alpha(x) \, dx = \int_\R K(x) \, dx$, which is independent of $\alpha$. Furthermore, the graph of $K_\alpha$ gets squeezed to the $y$-axis in the sense that $K_\alpha \weakto (\int_\R K(x) \, dx) \delta_0$ as $\alpha \to \infty$ in the weak topology of measures (i.e.\ tested against continuous and bounded functions). Hence, for $\alpha$ bounded away from $0$, the value in \eqref{pf:zk} scales as $n^2 / (\gamma \alpha)$. To make it of order $1$, we set $\gamma = 2 n^2 / \alpha$. 

The confinement part of $E^\alpha(\rho)$ is
\begin{equation*}
  \frac{ 2 \beta n }{\gamma} \int_\R Q(\alpha x) \, d\rho(x)
  = \frac{\beta \alpha}{n} \int_0^\alpha \frac1\alpha Q(y) \, dy
  = \frac{\beta}{n} P(\alpha).
\end{equation*}
For this to be of order $1$, we take $\alpha = P^{-1}(n/\beta) > q_2$. 

In conclusion, we take
\begin{equation} \label{alpha} 
  \alpha = P^{-1} \Big( \frac n\beta \Big),
  \qquad \gamma = 2 \frac{n^2}\alpha.
\end{equation} 
Substituting this into \eqref{E:def}, we obtain 
\begin{align} \label{Ealpha}
  E^\alpha (\rho) 
  &= \frac12 \int_\R \int_\R K_\alpha(x-y) \, d\rho (y) d\rho (x) + \int_\R Q_\alpha(x) \, d\rho(x),
\end{align}
where 
\begin{equation*} %\label{Qal}
    Q_\alpha (x) := \frac{\alpha}{P(\alpha)} Q ( \alpha x ).
\end{equation*}
The scaling in $\alpha$ is such that
\begin{equation} \label{Qa:normzn} 
  \int_0^1 Q_\alpha(x) \, dx
  = \frac1{P(\alpha)} \int_0^\alpha Q ( z ) \, dz
  = 1
  \qquad \text{for all } \alpha > q_2,
\end{equation}
and such that for the $p$-homogeneous $Q(x) = |x|^p$ with $p \geq 1$ we have (recall \eqref{P:example}) $Q_\alpha(x) = (p+1) |x|^p$, which is independent of $\alpha$. 

The expression in \eqref{Ealpha} shows that, as the notation suggests, $E^\alpha$ depends on $n$ only through $\alpha$. Therefore, also $\orho$ depends on $n$ only through $\alpha$. Moreover,
we show in Lemmas \ref{l:Ea:bds} and \ref{l:orhoa:supp} that, as desired, the mass of $\orho$ spreads out over a length scale of order $1$ and 
\begin{equation} \label{Ea:ULB}
  \frac1C \leq E^\alpha(\orho) \leq C
\end{equation}
for some constant $C > 0$ independent of $\alpha$.

In our reformulation of Theorem \ref{t:intro} we use $\alpha$ instead of $\beta$. By inverting \eqref{alpha} we recover
\begin{equation*} %\label{beta}
  \beta = \frac n{P(\alpha)}.
\end{equation*}
Moreover, \eqref{Pinv:ass} reads as
\begin{equation} \label{Pinv:ass:al}  
  \alpha 
  \geq - P^{-1} ( -P(\alpha) )
  \qquad \text{for all } \alpha > q_2.
\end{equation}

Next we apply the scaling \eqref{E:def},\eqref{alpha} to $I_n^D$. This yields
\begin{equation} \label{Ena:scaling}
  E_n^\alpha (\bx) 
  := \frac1\gamma I_n^D(\alpha \bx)
  = \frac1{2 n (n-1)} \sum_{i =1}^n \sum_{ \substack{ j = 1 \\ j \neq i }}^n K_\alpha (x_i - x_j) + \frac1n \sum_{i=1}^n Q_\alpha (  x_i ),
\end{equation}
where we consider $E_n^\alpha $ as a function defined on the domain
\[
  \Omega_n := \{ \bx = (x_1, \ldots, x_n) \in \R^n \mid x_1 < \ldots < x_n \}.
\]
Finally, we define similarly to \eqref{FD:FC} the potential values
\begin{equation} \label{Fan:Fa}
  F_n^\alpha := E_n^\alpha(\obx) - \frac1{2n} \sum_{i=1}^n Q_\alpha(\ox_i),
  \qquad
  F^\alpha := E^\alpha (\orho) - \frac12 \int_\R Q_\alpha(x) \, d\orho(x),
\end{equation}
where $\obx := \oba / \alpha$ is a minimizer of $E_n^\alpha$.

\begin{thm}  \label{t}
Let $Q$ satisfy Assumption \ref{a:Q}, \eqref{Q:min0} and \eqref{Pinv:ass:al}, and let $q_2$ be as in \eqref{J:q1:q2}. Then for any $c > 0$ there exist $C, C' > 0$ such that for all $n \geq 2$ and all $\max \{c, q_2 \} < \alpha \leq n$ there holds $\supp \orho \subset [-C, C]$,
\[
  -\sqrt{ A_n^\alpha }
   \leq F_n^\alpha - F^\alpha 
   \leq 0
\]
and 
\[
  - A_n^\alpha
   \leq E_n^\alpha(\obx) - E^\alpha (\orho) 
   \leq A_n^\alpha,
\] 
where
\[
  A_n^\alpha := C' \min \Big\{ \frac \alpha n \log \Big( \frac n\alpha \big( 1 + \| Q_\alpha'' \|_{L^\infty(\supp \orho)} \big) \Big), 1 \Big\}.
\] 
\end{thm} 

Next we give four comments on the statement of Theorem \ref{t}. First, the three estimates involving $A_n^\alpha$ are interesting when $A_n^\alpha$ is sufficiently small, i.e.\ when the first term in the minimum in the expression of $A_n^\alpha$ is significantly smaller than $1$. Otherwise, Theorem \ref{t} can be proven from a priori estimates (see the argument at the start of Section \ref{s:pf:t}), and the resulting bounds on the potential difference is qualitatively similar to \eqref{HT19:estimate:I}. Hence, Theorem \ref{t} is interesting for $\alpha \ll \frac n{\log n}$ as $n \to \infty$. The required bound $\alpha \leq n$ is chosen for convenience to cover this range.

Second, certain bounds on the energy difference have already been established in \cite{KimuraVanMeurs21}: %\comm{The overview of the main Theorem 1.2 in Sec 1.2 shows where to find the proof of the energy estimates, Sec 6 does the log case (with $(\log n)^3$), and Rem 5.4 shows that we may work with $(\log n)^1$}

\begin{thm}[{\cite{KimuraVanMeurs21}}]  \label{t:KvM21}
Under the conditions in Theorem \ref{t}, there exists a constant $D_\alpha > 0$ independent of $n$ and with continuous dependence on $\alpha$ such that
\[
  |E_n^\alpha(\obx) - E^\alpha (\orho)| \leq D_\alpha \frac{\log n}n.
\] 
\end{thm}

\noindent Therefore, when proving the bounds on the energy difference in Theorem \ref{t} we may assume that $\alpha \to \infty$ as $n \to \infty$. 

Third, in Theorem \ref{t} we exclude the case $\alpha \to 0$ as $n \to \infty$; we comment on this in Section \ref{s:main:disc}. 

Fourth, while we consider a specific interaction potential $K$ (see \eqref{eq:def_K}), we only use this explicit form for the upper bound on the potential difference. Our proof for the other three bounds easily extends to other potentials which are even, convex on $(0,\infty)$, singular at $0$ and integrable on $\R$ (such as the potential used in our application to plasticity). The corresponding constant $A_n^\alpha$ may be different.
\medskip

Next we comment on the proof of Theorem \ref{t}. The complete proof is given in Section \ref{s:pf:t}. The upper bound on the potential difference is given by Theorem \ref{t:TS19}. For the bounds on the energy difference $E_n^\alpha(\obx) - E^\alpha (\orho)$ we follow the steps of the proof of Theorem \ref{t:KvM21} in \cite{KimuraVanMeurs21}. The foundation of these steps are several quantitative properties of $\orho$ established \cite{KimuraVanMeurs20DOI}.  One of these properties is that $\| \orho \|_\infty \leq C_\alpha$ for some implicit constant $C_\alpha > 0$.  In our setting we require an explicit bound in terms of $\alpha$. The proof in \cite{KimuraVanMeurs20DOI} seems not suited for this, because it treats the logarithmic part of $K$ as the main contribution while the tails (here, the deviation of $K$ from the logarithm) are treated as a regular perturbation. As $\alpha$ gets large, the contribution of the tails becomes dominant, and thus treating them as a perturbation is expected to result in $C_\alpha$ to be unnecessary large. Therefore, we construct a new proof to bound $\| \orho \|_\infty$. The result is 
$$\| \orho \|_\infty \leq C(\| Q_\alpha'' \|_{L^\infty(\supp \orho)} + 1);$$ 
see Proposition \ref{prop:orhoa:sup:bd} below for a precise statement. This is the sole reason for the appearance of $Q_\alpha''$ in the estimates in Theorem \ref{t}. Furthermore, the proof of Proposition \ref{prop:orhoa:sup:bd} is the only place where we use the rather high regularity requirement $Q \in C^4(\R)$ (such that $\orho$ is of class $C^2$ inside its support); all other arguments carry through for $Q \in C^2(\R)$. The proof of Proposition \ref{prop:orhoa:sup:bd} is unconventional; we consider it the main mathematical contribution of this paper.

To appreciate that the upper bound on $\| \orho \|_\infty$ depends on $\alpha$ (and is actually finite), consider the example $Q(x) = e^{|x|^p} - 1$. From the (hyper)exponential growth and the normalization \eqref{Qa:normzn} it follows that
\[
  Q_\alpha(x) 
  \xto{\alpha \to \infty} Q_\infty(x) := \begin{cases}
    0
    &\text{if } |x| < 1 \\
    \infty
    &\text{if } |x| > 1
  \end{cases}  
\]
for all $|x| \neq 1$.
If we replace $Q_\alpha$ in $E^\alpha$ with $Q_\infty$, then it follows from \cite[Theorem 1.4]{KimuraVanMeurs20DOI} that the corresponding minimizer is not in $L^\infty(\R)$. Hence, we expect that for certain choices of $Q$ we have $\| \orho \|_\infty \to \infty$ as $\alpha \to \infty$.

Yet, the dependence of our bound on $\| \orho \|_\infty$ on $Q''$ is curious. We expect that it appears due to our unconventional proof, and that there exists a different bound which requires less regularity on $Q$.   We demonstrate in Section \ref{s:main:cor} that our current bound is sufficient for practical examples of $Q$, but also that there are pathological choices for which our bound seems suboptimal.

\subsection{Corollaries}
\label{s:main:cor}

In this section we present simplified statements of Theorem \ref{t} and Theorem \ref{t:intro} which are useful in practice. With this aim, for $a_n, b_n > 0$ let
\begin{align*}
&\text{``}a_n \ll b_n \text{ as } n \to \infty \text{"} 
&&\text{be defined as}  
&\lim_{n \to \infty} \frac{a_n}{b_n} &= 0, \\
&\text{``}a_n \sim b_n \text{ as } n \to \infty \text{"} 
&&\text{be defined as}  
&0 < \liminf_{n \to \infty} \frac{a_n}{b_n} &\leq \limsup_{n \to \infty} \frac{a_n}{b_n} < \infty, \\
&\text{``}a_n \lesssim b_n \text{ as } n \to \infty \text{"} 
&&\text{be defined as}  
&\limsup_{n \to \infty} \frac{a_n}{b_n} &< \infty.
\end{align*}
Whenever it is clear from the context, we omit ``as $n \to \infty$".

\begin{cor}[Simple version of Theorem \ref{t}]  \label{c:t}
Let $Q$ satisfy Assumption \ref{a:Q}, \eqref{Q:min0} and \eqref{Pinv:ass:al}. If $\| Q_\alpha'' \|_{L^\infty(\supp \orho)} \ll \alpha^M$ as $\alpha \to \infty$ for some $M > 0$, then there exists a $C > 0$ such that for $1 \ll \alpha \ll n$ as $n, \alpha \to \infty$ we have  
\[
  -C\sqrt{ \alpha \frac{\log n}n }
   \leq F_n^\alpha - F^\alpha 
   \leq 0
\]
and 
\[
  - C\alpha \frac{\log n}n
   \leq E_n^\alpha(\obx) - E^\alpha (\orho) 
   \leq C\alpha \frac{\log n}n.
\] 
\end{cor} 

\begin{cor}[Simple version of Theorem \ref{t:intro}] \label{c:t:intro}
Let $\beta = 1$ and let $Q$ satisfy Assumption \ref{a:Q}, \eqref{Q:min0} and \eqref{Pinv:ass}. If $\| Q'' \|_{L^\infty(\supp \omu)} \ll n^M$ as $n \to \infty$ for some $M > 0$, then there exists a $C > 0$ such that for all $n$ large enough we have
\[
  - C \sqrt{ \frac{n^3 \log n}{P^{-1}(n)}  }
   \leq F_n^D - F_n^C  
   \leq 0
\]
and 
\[
  - C n \log n
   \leq I_n^D(\oba) - I_n^C (\omu) 
   \leq C n \log n.
\] 
\end{cor} 

The proofs of these corollaries are obvious; by applying the imposed bounds on $Q''$ and $Q_\alpha''$ in the logarithms in the expressions for $A_n^\alpha$ and $B_n^\beta$, the values of these logarithms simplify to $C \log n$. Then, the bounds in the corollaries above follow directly.
\medskip

Next we demonstrate that for common choices of $Q$ the required bounds on $Q''$ and $Q_\alpha''$ are satisfied. We also provide a pathological example for which these  bounds fail. We start by noting that the bound on $Q''$ in Corollary \ref{c:t:intro} follows from that on $Q_\alpha''$ in Corollary \ref{c:t}. Indeed, by taking $M \geq 3$, $\alpha = P^{-1}(n)$ and recalling the properties of $P^{-1}$ listed in Section \ref{s:main:prel}, we observe from $1 \ll \alpha \lesssim \sqrt n$ that
\[
  \| Q'' \|_{L^\infty(\supp \omu)}
  = \frac{P(\alpha)}{\alpha^3} \| Q_\alpha'' \|_{L^\infty(\supp \orho)}
  \ll n \alpha^{M-3} \lesssim n^{\tfrac{M-3}2 + 1}.
\]
Hence, it is sufficient to verify the bound on $\| Q_\alpha'' \|_{L^\infty(\supp \orho)}$.

The first example is $Q(x) = |x|^p$ with $p \geq 2$ as in Section \ref{s:main:prel}. Recall that $Q_\alpha(x) = (p+1) |x|^p$ is independent of $\alpha$. Hence, using the bound on $\supp \orho$ from Theorem \ref{t}, we get
$$
  \| Q_\alpha'' \|_{L^\infty(\supp \orho)}
  = \| Q'' \|_{L^\infty(\supp \orho)}
  \leq \| Q'' \|_{L^\infty(-C, C)},
$$
where the right-hand side is a constant which only depends on $p$. Hence, the condition in Corollary \ref{c:t} is met for any $M > 0$. A similar conclusion is reached for  $1 \leq p < 2$ when $Q$ is regularized in a neighborhood around $0$. The argument is more subtle. Let $x_*$ be a maximizer of $Q''$. Then
$$
  \| Q_\alpha'' \|_{L^\infty(\supp \orho)}
  \leq \frac{ \alpha^3 }{P(\alpha)} Q''(x_*)
  \leq C \alpha^{2-p},
$$
and thus the condition in Corollary \ref{c:t} is met for any $M > 2-p$.

In the second example we consider $Q(x) = e^{|x|^p} - 1$ for $p \geq 2$. 
Using that $Q_\alpha''(x) \leq C' \alpha^{2(p-1)} (Q_\alpha(x) + 1)$ for all $x \in \supp \orho \subset [-C, C]$, 
we obtain that 
\[
  \| Q_\alpha'' \|_{L^\infty(\supp \orho)}
  \leq C \alpha^{2(p-1)} \| Q_\alpha + 1 \|_{L^\infty(\supp \orho)}.
\]
We show in Lemma \ref{l:orhoa:supp} that the right-hand side is proportional to $\alpha^{2(p-1)}$. Hence, the condition in Corollary \ref{c:t} is met for any $M > 2p - 2$. The case $1 \leq p < 2$ can be treated analogously as in the example above; a regularization around $0$ is required, and the resulting constant $M$ may need to be chosen larger. We omit the details.

Finally, we construct a pathological example of $Q$ for which the condition in Corollary \ref{c:t:intro} is not met. The idea is to take $Q$ as a carefully constructed regularization of a piecewise-linear function $\tilde Q$, i.e.\ $\tilde Q''$ is a sum of delta measures. Let $Q$ be defined by
\[
  Q(0) = Q'(0) = 0, \qquad
  Q''(x) = \sum_{k = 1}^\infty \big( \varphi_{e^{-k}} (x + k) + \varphi_{e^{-k}} (x - k) \big),
\]
where $\varphi_{\delta}$ is the usual mollifier. It is easy to verify that $Q$ satisfies Assumption \ref{a:Q}, \eqref{Q:min0} and \eqref{Pinv:ass}. Moreover, $Q$ is even and has quadratic growth, and thus $P(\alpha) \sim \alpha^3$. Relying on $\supp \omu = \alpha \supp \orho \supset [-c\alpha, c\alpha]$  for some $c > 0$ (see Lemma \ref{l:orhoa:supp}, and note from the evenness of $Q$ that $\orho$ is even) with $\alpha = P^{-1}(n) \sim \sqrt[3]{n}$, we get 
\[
  \| Q'' \|_{L^\infty(\supp \omu)}
  \geq \| Q'' \|_{L^\infty(-c\alpha, c\alpha)}
  \gtrsim \| \varphi_{e^{-c \alpha + 1}}'' \|_\infty
  \sim e^{3 (c \alpha - 1)}
  \gtrsim e^{c' \sqrt[3]{n}},
\]
which does not satisfy the required bound in Corollary \ref{c:t:intro}.

\subsection{Discussion}
\label{s:main:disc}

First we comment on the application of our main results (Theorems \ref{t:intro} and \ref{t} and their corresponding simplifications in Corollaries \ref{c:t} and \ref{c:t:intro}) to approximation theory and plasticity.
\smallskip

\textbf{Application to approximation theory.} 
Our goal was to improve the lower bound in \eqref{HT19:estimate:I}. Here, we compute the factor by which Corollary \ref{c:t:intro} improves this bound. Let $\alpha := P^{-1}(n)$ and recall that $1 \ll P^{-1}(n) \lesssim \sqrt n$ as $n \to \infty$. First, we note from the definition of $F_n^\alpha$ and $E^\alpha(\orho) \geq c > 0$ (recall \eqref{Ea:ULB}) that
\[
  2 F_n^\alpha 
  \geq E_n^\alpha(\obx) 
  \geq E^\alpha (\orho) - C \alpha \frac{\log n}n
  \geq c - C' \frac{\log n}{\sqrt n}
  \xto{n \to \infty} c.
\]
Hence, the left-hand side in \eqref{HT19:estimate:I} is in absolute value bounded from below by $c n^2 / P^{-1}(n)$ for all $n$ large enough. Thus, the bound in Corollary \ref{c:t:intro} improves the bound in \eqref{HT19:estimate:I} by a factor
\[
  C \sqrt{ \frac{n^3 \log n}{P^{-1}(n)}  } \bigg/ \bigg( c \frac{ n^2 }{ P^{-1}(n) } \bigg)
  = \frac Cc \sqrt{ P^{-1}(n) \frac{\log n}n }.
\]
The size of this improvement factor depends on the growth of $Q$. It ranges from $O(\sqrt{(\log n) / n})$ to $O ( \sqrt{(\log n) / \sqrt n} )$. 

As a corollary of the improved lower bound, we get better estimates on the error value in \eqref{eq:WCE}. Indeed, starting from \eqref{cEn:bds}, we obtain
\begin{multline*}
  - \frac{2n^2}{(n-1) P^{-1}(n)} F^\alpha
  = - \frac{ F_n^C }{n-1}
  \leq \log E_{n}^{\mathrm{min}}
  \leq - \frac{ F_n^D }n
  \leq - \frac{ F_n^C }n + C \sqrt{ \frac{n \log n}{P^{-1}(n)} } \\
  = \frac{2n}{P^{-1}(n)} \bigg(-  F^\alpha + \frac C2 \sqrt{ \frac{  \log n}{n} P^{-1}(n) }\bigg).
\end{multline*}
We have no explicit expression for $F^\alpha$ in terms of $n$. However, it is bounded away from $0$ and $\infty$; see \eqref{Ea:ULB} and recall that $\frac12 E^\alpha (\orho) \leq F^\alpha \leq E^\alpha (\orho)$. Hence, the difference between the right- and left-hand side is $O(\sqrt{ (n \log n)/P^{-1}(n)} )$.
\smallskip

\textbf{Application to plasticity.} In \cite{GeersPeerlingsPeletierScardia13}, five different descriptions were found for the limiting particle density  as $n \to \infty$ of the minimizer $\obx$ of an energy of the type \eqref{Ena:scaling}. Each of the five limiting particle densities corresponds to a different scaling regime of $\alpha = \alpha_n$. However, in practice, $n$ and $\alpha$ are given, finite modelling parameters, and thus bounds of the type ${\rm{dist}}  (\obx, \orho) \leq c_n^\alpha$ for explicit $c_n^\alpha \to 0$ as $n \to \infty$ are required to select the best fit out of the five limiting densities. For a distance concept which is related to the energy norm, the goal of bounding ${\rm{dist}}  (\obx, \orho)$ translates to finding upper and lower bounds on the energy difference in Corollary \ref{c:t}; see \cite{KimuraVanMeurs21}. With Corollary \ref{c:t} we have extended the previous bounds for $\alpha \sim 1$ (see Theorem \ref{t:KvM21}) to $1 \lesssim \alpha \ll \frac n{\log n}$.
\bigskip

Finally we comment on possible generalizations of Theorem \ref{t}, which translate directly to generalizations of Theorem \ref{t:intro}. First, the lower bound on the potential difference is weaker than that for the energy difference. We expect that it can be bounded from below by $-A_n^\alpha$, but we have no proof. Second, for the generic case $q_2 = 0$, it is possible to consider the regime $\alpha \ll 1$. In this case, the rescaling from $I_n^C$ to $E^\alpha$ has to be adjusted; see e.g.\ \cite{GeersPeerlingsPeletierScardia13,ScardiaPeerlingsPeletierGeers14}. Since we have no direct application in mind, we have not explored this case any further.

%---------------------------------
\section{Proof of Theorem \ref{t}}
\label{s:pf:t}

The upper bound on the potential difference $F_n^\alpha - F^\alpha$ is given by Theorem \ref{t:TS19}.
The three remaining bounds in Theorem \ref{t} are easy to prove when $A_n^\alpha$ is replaced by $C'$. Indeed, since $K, Q \geq 0$ we have
\[
  E_n^\alpha(\obx), E^\alpha (\orho) \geq 0.
\]
Furthermore, a priori bounds (see Lemma \ref{l:Ea:bds} below) provide a $C > 0$ independent of $n$ and $\alpha$ such that
\[
  E_n^\alpha(\obx), E^\alpha (\orho) \leq C.
\]
Finally, since $\frac12 E_n^\alpha(\obx) \leq F_n^\alpha \leq E_n^\alpha(\obx)$ and $\frac12 E^\alpha (\orho) \leq F^\alpha \leq E^\alpha (\orho)$, the three desired bounds follow. As a result, in the remainder we may assume that $A_n^\alpha$ is sufficiently small. Hence, it is sufficient to consider
$
  \alpha \leq \frac n{\log n}
$.

In addition to this upper bound on $\alpha$, we may assume by Theorem \ref{t:KvM21} that $\alpha \geq C$ for some constant $C > 0$ independent of $n$ when proving the bounds on the energy difference $E_n^\alpha(\obx) - E^\alpha (\orho)$.  

We prove the three remaining bounds of Theorem \ref{t} in Sections \ref{s:pf:t:E:UB}, \ref{s:pf:t:E:LB} and \ref{s:pf:t:F:LB}.  Prior to this, we cite (Section \ref{s:pf:t:prelim}) and establish (Section \ref{s:pf:t:orho}) preliminary results. In particular, in Section \ref{s:pf:t:orho} we prove new estimates on $\orho$ which are uniform in $\alpha$ and essential to the proof of Theorem \ref{t}.

\subsection{Preliminaries}
\label{s:pf:t:prelim}

\paragraph{Notation.} Here we mention some notation conventions. We reserve $c, C > 0$ for generic constants which do not depend on any of the relevant variables, in particular $n$ and $\alpha$. We use $C$ in upper bounds (and think of it as possibly large) and $c$ in lower bounds (and think of it as possibly small). While $c, C$ may vary from line to line, in each single display their value is fixed. If different constants appear in the same display, we denote them by $C, C', C''$ etc.

We set $B(x, r) := (x-r, x+r) \subset \R$ as the one-dimensional ball, $\cM(\R)$ as the space of signed, finite measures on $\R$ and $\cM_+(\R) \subset \cM(\R)$ as the subspace of non-negative measures.

To avoid clutter, we often omit the integration variable. For instance, for $\rho \in \cP(\R)$, we use
\begin{align*}
  \int_\R Q_\alpha \, d\rho &:= \int_\R Q_\alpha(x) \, d\rho(x) \\
  \int_{-1}^1 K_\alpha &:= \int_{-1}^1 K_\alpha(x) \, dx \\
  (K_\alpha * \rho)(x) &:= \int_\R K_\alpha(x - y) \, d\rho(y),
\end{align*}
and extrapolate this notation to other integrands.

\paragraph{Regularization of $K$.} First, we recall several properties of $K$ which we use throughout the proof of Theorem \ref{t}: 
\begin{itemize}
  \item $K \geq 0$ is even,
  \item on $(0,\infty)$, $K$ is decreasing and strictly convex,
  \item There exist $C_1, \ldots, C_4 > 0$ such that for all $x \neq 0$
\begin{align} \label{pf:zl}
  K(x) 
  &\leq \left\{ \begin{array}{ll}
    C_1 \log \dfrac1{|x|}
    & \text{if } |x| \leq \frac12 \\ 
    C_2
    & \text{otherwise}
  \end{array} \right. \\\notag
  |K'(x)| 
  &\leq \frac{C_3}{|x|} \\\notag
  0 \leq K''(x) 
  &\leq \frac{C_4}{x^2},
\end{align}
  \item There exists a $c > 0$ such that $x^2 K''(x) \geq c$ for all $x \in [-1,1] \setminus \{0\}$,
  \item There exists a $C > 0$ such that 
  \[ |K^{(k)}(x)| \leq C e^{-|x|} \quad \text{for all } |x| \geq 1 \text{ and } k =0,1,2. \]  
\end{itemize}

Next we recall the regularization of $K$ constructed in \cite{KimuraVanMeurs21}.  For fixed regularization parameter $\sigma > 0$, let 
\begin{equation} \label{KLM}
  L (x) := \left\{ \begin{aligned}
    & K(\sigma) + (x - \sigma) K'(\sigma)
    && \text{if } 0 \leq x \leq \sigma \\
    & K(x)
    && \text{if } x > \sigma,
  \end{aligned} \right.
  \qquad \qquad M:= K-L
\end{equation}
with even extension to the negative half-line. Figure \ref{fig:V:puntmuts} illustrates $L$. Lemma \ref{l:KLM} lists the relevant properties of $L$ and $M$.

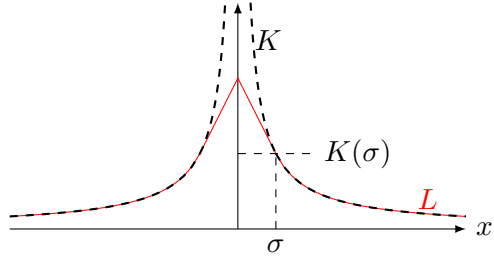
\begin{figure}[ht]
\centering
\begin{tikzpicture}[scale=0.5, >= latex]
\draw[dashed] (1,0) node[below] {$\sigma$} -- (1,2);
\draw[dashed] (0,2) --++ (2,0) node[right] {$K(\sigma)$};             
\draw[->] (-6,0) -- (6,0) node[right] {$x$};
\draw[->] (0,0) -- (0,6);
\draw[ red] (-1,2) -- (0,4) -- (1,2);
\draw[ domain=-6:-1, red] plot (\x,{-2/\x});
\draw[ domain=1:6, red] plot (\x,{2/\x});
\draw[thick, dashed, domain=-6:-1] plot (\x,{-2/\x});
\draw[thick, dashed, domain=0.3333:1] plot (\x,{2/\x});
\draw[thick, dashed, domain=-1:-0.3333] plot (\x,{-2/\x});
\draw[thick, dashed, domain=1:6] plot (\x,{2/\x});
\draw (.8,5) node {$K$};
\draw[red] (5,.8) node {$L$};
%\draw (2,4.3) rectangle (6,6);
%\draw (2.5,5.5) -- (4.5,5.5) node[right] {$K(x)$};
%\draw[ thick] (2.5,4.8) -- (4.5,4.8) node[right] {$L(x)$};
\end{tikzpicture} \\
\caption{The regularization $L$ of $K$ by piecewise-affine extension on $[-\sigma, \sigma]$.}
\label{fig:V:puntmuts}
\end{figure}

\begin{lem} [Properties of $L$ and $M$ {\cite[Lemma 5.2]{KimuraVanMeurs21}}] \label{l:KLM} 
There exists a constant $C > 0$ such that for all $\sigma > 0$:
\begin{enumerate}[label=(\roman*)]
  \item \label{l:KLM:cv} $L$ is convex on $(0,\infty)$; 
  \item \label{l:KLM:Lnorm} $(\mu, \nu)_L := \int_\R (L * \mu) \, d\nu$ is a degenerate inner product on $\cM(\R)$;
  \item \label{l:KLM:L0} $L(0) \leq C (|\log \sigma| + 1)$; 
  \item \label{l:KLM:Mint} $\int_\R M \leq C \sigma(|\log \sigma| + 1)$.   
\end{enumerate}
\end{lem}

\begin{proof} 
\ref{l:KLM:cv} is a direct consequence from the convexity of $K$. Convexity and integrability of $L$ imply \ref{l:KLM:Lnorm}; see the proofs of \cite[Lemmas 3.1 and 3.2]{KimuraVanMeurs20DOI}. Noting that \ref{l:KLM:L0} and \ref{l:KLM:Mint} are obvious for $\sigma$ bounded away from $0$, it suffices to prove them for $\sigma < \frac12$. For such $\sigma$, we have $K(\sigma) \leq C |\log \sigma|$ and $|K'(\sigma)| \leq C' / \sigma$. Then, \ref{l:KLM:L0} follows directly from \eqref{KLM}. Finally, \ref{l:KLM:Mint} follows from
\begin{align*}
  \int_\R M
  \leq \int_{-\sigma}^\sigma K
  = 2 \int_0^\sigma K
  \leq C \int_0^\sigma -\log x \, dx
  = C \sigma (|\log \sigma| + 1).
\end{align*}
\end{proof}

\paragraph{Properties of $\orho$.} The energy $E^\alpha$ fits to the setting of \cite[Theorem 6.5]{KimuraVanMeurs20DOI} (except for a trivial rescaling to get $K$ of unit integral). This theorem states that for $\alpha > q_2$ fixed: \comm{This theorem does not care about $\alpha$, but for $\alpha \leq q_2$ the function $Q_\alpha$ is not defined}
\begin{itemize}
  \item there exists a unique minimizer $\orho \in \cP(\R)$ of $E^\alpha$,
  \item there exist $y_1^\alpha < y_2^\alpha$ such that 
  \[
    \supp \orho = [y_1^\alpha, y_2^\alpha],
  \]
  \item $\orho$ is absolutely continuous as a measure, and its density (also denoted by $\orho$) is H\"older continuous on $\R$,
  \item since $Q_\alpha \in C^4(\R) \subset W_{\rm loc}^{4,\infty} (\R)$, we have $\orho \in W^{3,p} (y_1^\alpha, y_2^\alpha) \subset C^2 ((y_1^\alpha, y_2^\alpha))$ for some $p > 1$,
  \item there exists $C_\alpha > 0$ such that $\displaystyle \frac{ \orho (x) }{ \sqrt{ (x - y_1^\alpha)(y_2^\alpha - x) } } \leq C_\alpha$ for any $x \in (y_1^\alpha, y_2^\alpha)$,
  \item $\orho$ satisfies the Euler-Lagrange equation given by
  \begin{subequations} \label{EL}
  \begin{align}\label{EL:eq}
    K_\alpha * \orho + Q_\alpha 
    &= F^\alpha
    &&\text{on } [y_1^\alpha, y_2^\alpha] \\ %[y_1^\alpha, y_2^\alpha] \\
    \label{EL:ineq:on:R}
    K_\alpha * \orho + Q_\alpha 
    &\geq F^\alpha
    &&\text{on } \R,
  \end{align}
  \end{subequations}
  where $F^\alpha$ is the potential value defined in \eqref{Fan:Fa}.
\end{itemize} 
We recall that, since the dependence of $\orho$ on $\alpha$ is obvious, we do not explicitly adopt this dependence in the notation.

In the remainder of Section \ref{s:pf:t} we treat $\orho$ as a function from $\R$ to $[0,\infty)$ and restrict the domain of $E^\alpha$ from measures to functions. As a result, we write
\[
  E^\alpha (\rho) 
  = \frac12 \int_\R (K_\alpha * \rho) \rho + \int_\R Q_\alpha \rho.
\]

\subsection{Uniform bounds on $\orho$}
\label{s:pf:t:orho}

While the properties of $\orho$ established in \cite{KimuraVanMeurs20DOI} and listed in Section \ref{s:pf:t:prelim} are essential for proving Theorem \ref{t}, they are not sufficient for our proof. In particular, these properties give no control on the related values of $y_1^\alpha, y_2^\alpha, C_\alpha, E^\alpha(\orho)$ for large $\alpha$. To find sufficient uniform bounds on these and related quantities, we apply and develop different arguments.

We start with two auxiliary lemmas: 

\begin{lem}[Lower bound on the interaction energy] \label{l:LB:Eint}
For any $\rho \in L^1(\R) \cap \cM_+(\R)$ and any $z_1 < z_2$, we have for all $\alpha > 0$
\[
  \int_\R (K_\alpha * \rho) \rho 
  \geq \frac{K(1)}{z_2 - z_1} \bigg( \int_{z_1}^{z_2} \rho \bigg)^2.
\]
\end{lem}

\begin{proof}
Since $K \geq K(1) \chi_{[-1,1]} \geq 0$ (recall that $\chi$ is the indicator function defined in \eqref{chi}),
\begin{equation} \label{pf12}
  \int_\R (K_\alpha * \rho) \rho
  \geq \alpha K(1) \iint_{[z_1, z_2]^2} \chi_{[-1,1]} (\alpha( x-y)) \rho(y) \rho(x) \, dydx.
\end{equation}
We reduce the integration domain to the union of disjoint squares $S_i := [z_1 + \frac{i-1}\alpha, z_1 + \frac i\alpha)^2$ along the diagonal, where $1 \leq i \leq N_\alpha$ and $N_\alpha := \alpha (z_2 - z_1)$. We assume that
$
  N_\alpha \in \N
$
and treat the general case afterwards. Note that $\chi_{[-1,1]} (\alpha( x-y)) = 1$ for all $(x,y) \in S_i$ and each $i$. Then, for the integral in the right-hand side of \eqref{pf12}, we estimate
\begin{multline*}
  \iint_{[z_1, z_2]^2} \chi_{[-1,1]} (\alpha( x-y)) \rho(y) \rho(x) \, dydx
  \geq \sum_{i=1}^{N_\alpha} \iint_{S_i} \rho(y) \rho(x) \, dydx \\
  = \sum_{i=1}^{N_\alpha} \bigg( \int_{z_1 + \frac{i-1}\alpha}^{z_1 + \frac i\alpha} \rho \bigg)^2
  \geq \frac1{N_\alpha} \bigg( \sum_{i=1}^{N_\alpha}  \int_{z_1 + \frac{i-1}\alpha}^{z_1 + \frac i\alpha} \rho \bigg)^2
  = \frac1\alpha \frac1{z_2 - z_1} \bigg( \int_{z_1}^{z_2} \rho \bigg)^2.
\end{multline*}
This proves Lemma \ref{l:LB:Eint} for when $
  N_\alpha = \alpha (z_2 - z_1) \in \N
$. 

If $\alpha (z_2 - z_1) \notin \N$, then the argument above holds with the following minor modification. Set 
$N_\alpha := \lceil \alpha (z_2 - z_1) \rceil$, and redefine the last square as
\[
  S_{N_\alpha} := \Big[ z_1 + \frac{N_\alpha - 1}\alpha, z_2 \Big)^2.
\]
Note that the size of $S_{N_\alpha}$ is less then $\frac1\alpha$. Then, all steps above apply verbatim.
\end{proof}

\begin{lem}[Lower bound on the interaction energy] \label{l:LB:Qa}
For all $\alpha > q_2$ and all $|x| \geq 1$ there holds
\begin{equation*} %\label{Qa:LB} 
  Q_\alpha(x) \geq 2 |x|. 
\end{equation*}
\end{lem}

\begin{proof} Note that the line parametrized by $x \mapsto Q_\alpha(1) x$ intersects the graph of $Q_\alpha$ at $x = 0$ and $x = 1$. Then, since $Q_\alpha$ is convex, we obtain $Q_\alpha(x) \leq Q_\alpha(1) x$ for all $0 \leq x \leq 1$ and $Q_\alpha(x) \geq Q_\alpha(1) x$ for all $x \geq 1$. Recalling \eqref{Qa:normzn}, we have
\[
  1 = \int_0^1 Q_\alpha(x) \, dx \leq \int_0^1 Q_\alpha(1) x \, dx = \frac12 Q_\alpha(1).
\]
Thus, for $x \geq 1$ we have $Q_\alpha(x) \geq Q_\alpha(1) x \geq 2x$; Lemma \ref{l:LB:Qa} follows for $x \geq 1$. Similarly, we obtain
\[
  Q_\alpha(-1) 
  = 2 \int_{-1}^0 Q_\alpha(-1) x \, dx
  \geq 2 \int_{-1}^0 Q_\alpha(x) \, dx
  = 2\frac{-P(-\alpha)}{P(\alpha)}
  \geq 2, 
\] 
where the last inequality follows from \eqref{Pinv:ass:al}. Then, for $x \leq -1$, we have $Q_\alpha(-x) \geq Q_\alpha(-1) |x| \geq 2|x|$. This completes the proof.
\end{proof}

Next we use the auxiliary Lemmas \ref{l:LB:Eint} and \ref{l:LB:Qa} to establish uniform bounds on the energies.

\begin{lem}[Bounds on $E_n^\alpha(\obx), E^\alpha(\orho)$] \label{l:Ea:bds}
There exists a $C > 1$ such that for all $n \geq 2$ and all $\alpha > q_2$
\[
  E_n^\alpha(\obx), E^\alpha(\orho) \leq C
\]
and
\[
   E^\alpha(\orho) \geq \frac1C.
\]
\end{lem}

\begin{proof}
Let $\rho := \chi_{[0,1]} \in \cP(\R)$ and note that 
\[
  \| K_\alpha * \rho \|_\infty
  \leq \| \rho \|_\infty \int_\R K_\alpha
  = \int_\R K.
\]
Then, by the minimality of $\orho$ and \eqref{Qa:normzn},
\[
  E^\alpha(\orho)
  \leq E^\alpha(\rho)
  \leq \frac12 \| K_\alpha * \rho \|_\infty \int_\R \rho + \int_0^1 Q_\alpha \leq C.
\]
Similarly, for $E_n^\alpha(\obx)$ we take $\bx$ defined by $x_i = \frac{i-1}n$ for $i = 1, \ldots,n$. Then, 
\begin{align*}  
  E_n^\alpha(\obx)
  \leq E_n^\alpha(\bx)
  = \frac1{2 n (n-1)} \sum_{i =1}^n \sum_{ \substack{ j = 1 \\ j \neq i }}^n K_\alpha \Big( \frac{i-j}n \Big) +  \frac1n \sum_{i=1}^n Q_\alpha \Big( \frac{i-1}n \Big).
\end{align*}
Recalling that $K_\alpha$ is even and that $K_\alpha > 0$ and $Q_\alpha(0) = 0$, we estimate
\begin{align*}  
  E_n^\alpha(\obx)
  \leq \frac1{n-1} \sum_{i =1}^n \sum_{ k=1 }^{n-1} \frac1n K_\alpha \Big( \frac{k}n \Big) + \sum_{j=1}^{n-1} \frac1n Q_\alpha \Big( \frac jn \Big).
\end{align*}
From the monotonicity properties of $Q$ and $K$ we recognize the sums above as lower Riemann sums of integrals over $K_\alpha$ and $Q_\alpha$. Using this, we obtain
\begin{align*}  
  E_n^\alpha(\obx)
  \leq \frac n{n-1} \int_0^1 K_\alpha + \int_0^1 Q_\alpha 
  \leq C.
\end{align*}

To prove the lower bound in Lemma \ref{l:Ea:bds}, let
\begin{equation*} %\label{pf:zm}  
  \orho_\text{in} := \orho |_{[-1, 1]}
  \quad \text{and} \quad
  \orho_\text{out} := \orho - \orho_\text{in} = \orho |_{\R \setminus [-1, 1]}.
\end{equation*}
Set $m := \int_\R \orho_\text{in} \in [0,1]$, and note that $\int_\R \orho_\text{out} = 1-m$. 
Then, by Lemma \ref{l:LB:Qa}
\[
  \int_\R Q_\alpha \orho
  \geq \int_\R Q_\alpha \orho_\text{out}
  \geq \int_\R 2 \orho_\text{out} = 2(1 - m),
\]
and by Lemma \ref{l:LB:Eint}
\[
  \int_\R (K_\alpha * \orho) \, \orho 
  \geq \int_\R (K_\alpha * \orho_\text{in}) \, \orho_\text{in} 
  \geq \frac{K(1)}2 \bigg( \int_{-1}^{1} \orho_\text{in} \bigg)^2
  = \frac{K(1)}2 m^2.
\]
Hence,
\[
  E^\alpha(\orho) 
  \geq \frac{K(1)}4 m^2 + 2(1-m)
  \geq c > 0 
\]
for some $c$ independent of $m$.
\end{proof}

For the following lemma, we recall that $\supp \orho = [y_1^\alpha, y_2^\alpha]$.

\begin{lem}[Bounds on $\supp \orho$] \label{l:orhoa:supp}
There exist $C,c > 0$ such that 
\[
  - C 
  \leq y_1^\alpha 
  \leq 0
  \leq y_2^\alpha
  \leq C
  \quad \text{ and } \quad
  y_2^\alpha - y_1^\alpha > c
  \quad \text{for all } \alpha > q_2.
\]
Moreover, there exists a $C > 0$ such that 
\begin{equation} \label{l:orhoa:supp:Qn:UB} 
  \| Q_\alpha \|_{L^\infty( y_1^\alpha, y_2^\alpha )} \leq C
  \quad \text{for all } \alpha > q_2.
\end{equation}
\end{lem}

\begin{proof}
We simplify notation by setting $y_i = y_i^\alpha$ ($i=1,2$). 

It is easy to see that $0 \in [y_1, y_2]$. Indeed, if this is not the case, then by the monotonicity of $Q$ on either side of $0$ it follows that $E^\alpha$ is non-increasing along a translation of $\orho$ towards $0$; this contradicts the uniqueness of the minimizer $\orho$.

Next we prove $y_2 - y_1 \geq c$. Using Lemma \ref{l:LB:Eint}, we obtain
\begin{equation*}
  E^\alpha(\orho)
  \geq \frac{K(1)}{2(y_2 - y_1)} \bigg( \int_{y_1}^{y_2} d\orho \bigg)^2
  = \frac{K(1)}{2(y_2 - y_1)}.
\end{equation*}
Then, by Lemma \ref{l:Ea:bds}
\[
  y_2 - y_1
  \geq \frac {K(1)}{2 E^\alpha(\orho)} \geq c > 0.
\]

Finally, we claim that
\begin{equation} \label{pf15} 
  [y_1, y_2] \subset \{ Q_\alpha \leq 2 (E^\alpha(\orho) + 1) \} =: J,
\end{equation} 
where $J$ is a sublevel set of $Q_\alpha$. Given this claim, \eqref{l:orhoa:supp:Qn:UB} follows from $E^\alpha(\orho) \leq C$. Moreover, from the convexity and the lower bound on $Q_\alpha$ in Lemma \ref{l:LB:Qa} it follows that $J$ is a closed interval which is uniformly bounded in $\alpha$. Hence, given that \eqref{pf15} holds, $y_1$ and $y_2$ are uniformly bounded, which completes the proof of Lemma \ref{l:orhoa:supp}.

It is left to prove the claim \eqref{pf15}. We follow the proof by contradiction of \cite[Section 3.2]{MoraRondiScardia19} and \cite[Proposition 3.1]{KimuraVanMeurs21} with modifications. Suppose $[y_1, y_2] \not \subset J$. We are going to reach a contradiction by constructing a $\rho \in L^1(\R) \cap \cP(\R)$ such that $E^\alpha(\rho) < E^\alpha(\orho)$. In preparation for this, take $\ov m := \int_{J} \orho < 1$. To see that $\ov m > 0$, we rely on Lemma \ref{l:LB:Qa} and \eqref{pf15} to estimate
\begin{align*}
  E^\alpha(\orho)
  \geq \int_{J^c} Q_\alpha \orho
  >  2 (E^\alpha(\orho) + 1) \int_{J^c} \orho
  = 2 (E^\alpha(\orho) + 1) (1 - \ov m).
\end{align*} 
Hence, $\ov m > \frac12 > 0$. This allows us to define
\[
  \rho := \frac1{\ov m} \orho |_J \in L^1(\R) \cap \cP(\R).
\]
We rely on \eqref{pf15} to obtain
\begin{align*}
  E^\alpha(\orho)
  &= \frac12 \iint_{\R^2 \setminus J^2} \big[ K_\alpha(x-y) + Q_\alpha(x) + Q_\alpha(y) \big] \orho(y) \orho(x) \, dydx \\
  &\qquad + \frac12 \iint_{J^2} \big[ K_\alpha(x-y) + Q_\alpha(x) + Q_\alpha(y) \big] \ov m^2 \rho(y) \rho(x) \, dydx \\
  &\geq \frac12 \iint_{\R^2 \setminus J^2} 2 (E^\alpha(\orho) + 1) \orho(y) \orho(x) \, dydx + \ov m^2 E^\alpha(\rho) \\
  &= (1 - \ov m^2) (E^\alpha(\orho) + 1) + \ov m^2 E^\alpha(\rho).
\end{align*}
Rearranging terms, we obtain
\[
  \ov m^2 E^\alpha(\rho) 
  \leq \ov m^2 E^\alpha(\orho) + \ov m^2 - 1
  < \ov m^2 E^\alpha(\orho).
\]
\end{proof}

Finally we are ready to establish the key estimate in Proposition \ref{prop:orhoa:sup:bd}.

\begin{prop}[Upper bound on $\orho$] \label{prop:orhoa:sup:bd}
There exists a $C > 0$ such that 
\[
  \| \orho \|_\infty \leq C ( \| Q_\alpha'' \|_{L^\infty(\supp \orho)}  + 1) \quad
  \text{for all } \alpha > \max \{ q_2, 1 \}.
\]
\end{prop}

The proof relies on the following technical lemma, which we prove in Appendix \ref{a:pf:l:convo}.

\begin{lem} \label{l:convo:pp}
Let $\alpha > 0$, $a < b$ and $f \in L^1(\R)$. If $f|_{(a,b)} \in C^2((a,b))$, then $K_\alpha * f$ is two times differentiable on $(a,b)$, and for all $x \in (a,b)$
\[
  (K_\alpha * f)''(x)
  = \int_{\R} \big[ f(x + z) - f(x) - z f'(x) \big] \, K_\alpha''(z) dz.
\]
\end{lem}
\comm{"$\int_\R$" is correct here. The convolution on the left is also over $\R$. Note that $x$, however, is in  $(a,b)$.}

To see that the integral in the right-hand side is well-defined despite the singularity of $K_\alpha''(z)$ of type $1/z^2$, note that the term inside the square brackets is the second-order remainder term of Taylor's Theorem applied to $f$ at $x$, which is $O(z^2)$.

\begin{proof}[Proof of Proposition \ref{prop:orhoa:sup:bd}] 
We simplify notation by setting $y_i = y_i^\alpha$ ($i=1,2$). The properties of $\orho$ listed in Section \ref{s:pf:t:prelim} imply Proposition \ref{prop:orhoa:sup:bd} for $\alpha \leq C$ for some constant $C > 0$. Hence, it remains to prove Proposition \ref{prop:orhoa:sup:bd} for all $\alpha$ large enough. \comm{It is here that we need $\alpha > \max \{q_2, 1\}$ in the case $q_2 = 0$. Also, the resulting constant $C$ in Proposition \ref{l:convo:pp} depends on the $C$ in $\alpha \leq C$. }

By the Euler-Lagrange equation \eqref{EL},
\begin{equation} \label{pf1}
  - \| Q_\alpha'' \|_{L^\infty(\supp \orho)} 
  \leq - Q_\alpha''(x) = (K_\alpha * \orho)''(x)
  \qquad \text{for all } x \in (y_1,y_2).
\end{equation}
Since $\orho \in C^2((y_1,y_2))$, Lemma \ref{l:convo:pp} applies. This yields
\begin{equation} \label{pf2}
  (K_\alpha * \orho)''(x)
  = \int_{\R} \big[ \orho(x + z) - \orho(x) - z \orho'(x) \big] \, K_\alpha''(z) dz
  \qquad \text{for all } x \in (y_1,y_2).
\end{equation}

Next we briefly sketch the remainder of the proof. While \eqref{pf1} and \eqref{pf2} hold for all $x \in (y_1, y_2)$, we apply them only at a single point $x_* \in (y_1,y_2)$. We choose $x_*$ such that we can squeeze a parabola $q$ in between $\orho$ and its tangent line $\varphi(x) := \orho(x_*) + (x - x_*) \orho'(x_*)$ at $x = x_*$; see Figure \ref{fig:orho:UB}. Then, we can bound \eqref{pf2} at $x = x_*$ from above by replacing the term in brackets by $q - \varphi$. Except for the values $\orho(x_*)$ and $\orho'(x_*)$, this eliminates all dependence on $\orho$, and the estimate in Proposition \ref{prop:orhoa:sup:bd} follows by some simple algebra. 

We continue with the proof. As illustrated in Figure \ref{fig:orho:UB}, let
\[
  m := \frac{y_2 + y_1}2
  \quad \text{and} \quad
  \ell := y_2 - y_1
\]
be respectively the midpoint and length of $\supp \orho = [y_1, y_2]$. Then, we set
\[
  \tilde q(x) := - \| \orho \|_\infty \Big( \frac{x - m}\ell \Big)^2
\]
and
\begin{equation} \label{pf4}
  \Gamma
  := \max_{[y_1,y_2]} (\orho - \tilde q)
  \geq \| \orho \|_\infty.
\end{equation}
Then,
\[
  q(x)
  := \Gamma + \tilde q(x)
  = \Gamma - \| \orho \|_\infty \Big( \frac{x - m}\ell \Big)^2
\]
satisfies
\begin{equation} \label{pf7}
  \min_{[y_1,y_2]} (q - \orho) 
  = - \max_{[y_1,y_2]} (\orho - \tilde q - \Gamma)
  = 0.
\end{equation}

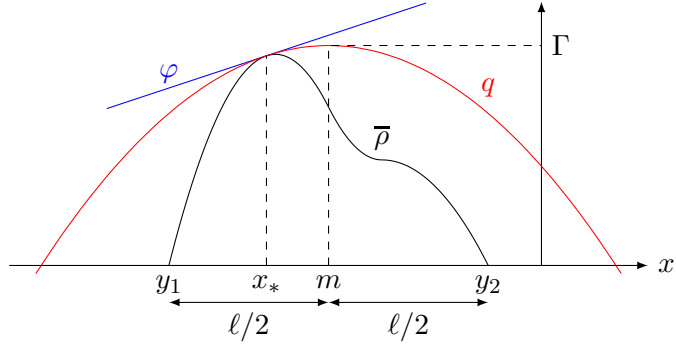
\begin{figure}[ht]
\centering
\begin{tikzpicture}[scale=.7, >= latex]
%\draw[green] (0,0) grid (6,5);
\draw[->] (7,0) --++ (0,5);
\draw[->] (-3,0) -- (9,0) node[right] {$x$};

\begin{scope}[shift={(2,4)},scale=1] 
    \draw[domain=-2:1, smooth] plot (\x,{ -\x*\x });
\end{scope}
\begin{scope}[shift={(4,2)},scale=1] 
    \draw[domain=-1:0, smooth] plot (\x,{ \x*\x });
    \draw[domain=0:2, smooth] plot (\x,{ -\x*\x/2 });
\end{scope}

\begin{scope}[shift={(11/6, 143/36)},scale=1] 
    \draw[domain=-3:3, smooth, blue] plot (\x,{ \x/3 });
\end{scope}

\begin{scope}[shift={(3, 25/6)},scale=1] 
    \draw[domain=-5.5:5.5, smooth, red] plot (\x,{ -\x*\x/7 });
\end{scope}

\draw (0,0) node[below] {$y_1$};
\draw (6,0) node[below] {$y_2$};
\draw[dashed] (11/6,0) node[below] {$x_*$} --++ (0,143/36);
\draw[blue] (0,4) node[below]{$\varphi$};
\draw (4,2) node[above] {$\orho$};
\draw[red] (6,3) node[above] {$q$};
%\draw (3,5) node[anchor = south east] {$M$};
%\draw[dashed] (2,4) --++ (5,0) node[anchor = south west] {$\Gamma$};
\draw[dashed] (3,0) node[below] {$m$} --++ (0,25/6) --++ (4,0) node[right] {$\Gamma$};
\draw[<->] (0,-.7) --++ (3,0) node[midway, below] {$\ell/2$};
\draw[<->] (3,-.7) --++ (3,0) node[midway, below] {$\ell/2$};
\end{tikzpicture} \\
\caption{Sketch of the functions $\orho$, $\varphi$, $q$ with corresponding constants $y_1$, $y_2$, $m$, $\ell$, $x_*$, $\Gamma$.}
\label{fig:orho:UB}
\end{figure}

Let $x_*$ be a minimizer of \eqref{pf7}, i.e.\ 
\begin{equation} \label{pf8}
  q(x_*) = \orho (x_*).
\end{equation}
Since for $x \in \{y_1,y_2\}$ we have 
\[
  (q - \orho)(x) 
  = \Gamma - \Big( \frac{x-m}\ell \Big)^2 \| \orho \|_\infty
  = \Gamma - \frac14 \| \orho \|_\infty
  \geq \frac34 \| \orho \|_\infty > 0,
\]
if follows from \eqref{pf7} that $x_* \in (y_1,y_2)$. Hence, by minimality of $x_*$,
\begin{equation} \label{pf9}
  q'(x_*) = \orho' (x_*).
\end{equation}
Finally, since $\orho(x) = 0$ for $x \notin (y_1,y_2)$, \eqref{pf7} implies
\begin{equation} \label{pf6}
  \min_{ \{ q \geq 0 \} } (q - \orho) 
  = 0.
\end{equation}
In conclusion, from \eqref{pf8}, \eqref{pf9} and \eqref{pf6} we obtain the following upper bound for the term inside brackets in \eqref{pf2} at $x = x_*$:
\begin{multline} \label{pf3}
  q(x_* + z) \geq 0
  \implies \\ 
  \oR_*(z) 
  := \orho(x_* + z) - \orho(x_*) - z \orho'(x_*) 
  \leq q(x_* + z) - q(x_*) - z q'(x_*)
  = - \| \orho \|_\infty \frac{z^2}{\ell^2}.
\end{multline}

Next we construct a sufficient condition for $z \in \R$ to satisfy the condition $q(x_* + z) \geq 0$. Recalling \eqref{pf4}, we observe that
\[
  r(x)
  := \| \orho \|_\infty \Big( 1 - \Big( \frac{x - m}\ell \Big)^2 \Big)
  \leq q(x).
\]
Since the zeros of the parabola $r$ are $m \pm \ell$, we have that $r > 0$ on $B(m, \ell)$, and thus 
\[
  B \Big( 0, \frac\ell2 \Big) \subset B(m - x_*, \ell) \subset \{ z \mid q(x_* + z) \geq 0 \}.
\]
Then, we obtain from \eqref{pf3} that 
\begin{align*} %\label{pf5}
  \oR_*(z)
  \leq - \| \orho \|_\infty \frac{z^2}{\ell^2}
  \qquad \text{for all } |z| \leq \frac\ell2.
\end{align*}

For $|z| \geq \frac\ell2$, we observe that $\oR_*$ is either negative or affine with slope bounded in absolute value by
\[
  |\orho'(x_*)|
  = |q'(x_*)|
  = 2 \| \orho \|_\infty \frac{|x_* - m|}{\ell^2}
  \leq \frac1\ell \| \orho \|_\infty.
\]
Hence,
\begin{align*}
  \oR_*(z) &\leq \frac1\ell \| \orho \|_\infty \Big( |z| - \frac\ell2 \Big) 
  && \text{for all } |z| \geq \frac\ell2.
\end{align*}
Finally, collecting the upper bounds on $\oR_*(z)$ above, we continue the estimate in \eqref{pf2} for $x = x_*$ by
\begin{align} \notag
  (K_\alpha * \orho)''(x_*)
  &= \int_{\R} \oR_*(z) K_\alpha''(z) \, dz \\\label{pf5}
  &\leq \frac1{\ell^2} \| \orho \|_\infty \bigg( - \int_{- \ell/2}^{\ell/2} z^2 K_\alpha''(z) \, dz 
    + \ell \int_{B(0, \ell/2)^c} \Big( |z| - \frac\ell2 \Big) K_\alpha''(z) \, dz \bigg).
\end{align} 

To estimate \eqref{pf5}, note that the integrands in both integrals are nonnegative. We bound each integral separately. By Lemma \ref{l:orhoa:supp}, $\ell \geq c > 0$ uniformly in $\alpha$. 
%\kt{$\leftarrow \alpha$?}
Then, for $\alpha \geq 2$ we obtain
\[
  \int_{- \ell/2}^{\ell/2} z^2 K_\alpha''(z) \, dz
  = \int_{- \alpha \ell/2}^{ \alpha \ell/2} y^2 K''(y) \, dy
  \geq \int_{-c}^c y^2 K''(y) \, dy = c'
  > 0.
\]
For the second integral in \eqref{pf5}, we integrate by parts to obtain
\begin{multline*}
  \int_{B(0, \ell/2)^c} \Big( |z| - \frac\ell2 \Big) K_\alpha''(z) \, dz
  = 2 \int_{\ell/2}^\infty \Big( z - \frac\ell2 \Big) K_\alpha''(z) \, dz
  = 2 \int_0^\infty \zeta K_\alpha''\Big( \zeta + \frac\ell2 \Big) \, d\zeta
  = 2 K_\alpha \Big( \frac\ell2 \Big) \\
  \leq 2 K_\alpha \Big( \frac{c}2 \Big)
  = 2 \alpha K \Big( \frac{c}2 \alpha \Big)
  \leq C \alpha e^{-c \alpha/2}
  \leq C' e^{- c' \alpha}.
\end{multline*}
Then, for all $\alpha$ large enough, we obtain from \eqref{pf1}, \eqref{pf5} and $c \leq \ell \leq C$ (see Lemma \ref{l:orhoa:supp}) that
\[
  -\| Q_\alpha'' \|_\infty
  \leq (K_\alpha * \orho)''(x_*)
  \leq \frac1{\ell^2} \| \orho \|_\infty \Big( - c +  C \ell e^{-c' \alpha} \Big)
  \leq - \frac c2 \| \orho \|_\infty.
\]
\end{proof}

\subsection{Upper bound on the energy difference}
%\subsection{Upper bound on $E_n^\alpha (\obx) - E^\alpha (\orho)$}
\label{s:pf:t:E:UB} 

Thanks to Lemma \ref{l:orhoa:supp} and Proposition \ref{prop:orhoa:sup:bd} we have obtained enough control on the dependence of $\orho, y_1^\alpha, y_2^\alpha$ on $\alpha$. We therefore simplify notation by setting $y_i = y_i^\alpha$ ($i=1,2$) in this and the following sections.

In this section we prove the following of the four inequalities in Theorem \ref{t}:
\begin{align} \label{pf:zp}
  E_n^\alpha (\obx) - E^\alpha (\orho) 
   \leq C \frac{\alpha}n \log \frac{ q_\alpha n }\alpha,
\end{align}
where 
\begin{align} \label{qal} 
  q_\alpha := \| Q_\alpha'' \|_{L^\infty(\supp \orho)}  + 1.  
\end{align}
The proof below is inspired by \cite[Section 4]{KimuraVanMeurs21}. 
 
First, we introduce a small modification of $E_n^\alpha$. Let $\hat E_n^\alpha : \R^{n+1} \to [0,\infty)$ be defined by
\[
  \hat E_n^\alpha (\hat \bx) := \frac1{2 n^2} \sum_{i =0}^n \sum_{ \substack{ j = 0 \\ j \neq i }}^n K_\alpha (\hat x_i - \hat x_j)  + \frac1n \sum_{i=0}^n Q_\alpha (  \hat x_i ).
\]
We will construct $\hat \bx \in \R^{n+1}$ such that
\begin{align} \label{pf:zs}
  \hat E_n^\alpha (\hat \bx) - E^\alpha (\orho)
  \leq C \frac \alpha n \log \frac{ q_\alpha n }\alpha.
\end{align}
By the minimality of $\obx$ it is easy to see that
\[
  E_n^\alpha (\obx) 
  \leq \hat E_n^\alpha (\hat \bx) + \frac1{n-1} \hat E_n^\alpha (\hat \bx).
\]
Then, \eqref{pf:zp} follows by applying \eqref{pf:zs} to the second term in the right-hand side and by recalling from Lemma \ref{l:Ea:bds} that $E^\alpha (\orho) \leq C$.
\smallskip

To prove \eqref{pf:zs}, we take $\hat \bx$ such that
\[
  \hat x_0 = y_1,
  \qquad \hat x_n = y_2,
  \qquad \int_{\hat x_0}^{\hat x_i} \orho = \frac in \quad \text{for all } 1 \leq i \leq n-1.
\]
It is easy to see that this defines a unique $\hat \bx \in \R^{n+1}$, and that
\begin{align} \label{pf:zr}
  \int_{\hat x_{i-1}}^{\hat x_i} \orho = \frac1n
  \quad \text{for all } 1 \leq i \leq n.
\end{align}
Related to $\hat \bx$, we set for $i = 1,\ldots, n$
\begin{align*}
  I_i &:= [\hat x_{i-1}, \hat x_i], \\
  \ell_i &:= \hat x_i - \hat x_{i-1} = |I_i|, \\
  \varphi_i &:= \frac1{n \ell_i } \chi_{I_i}, \\
  \varphi &:= \sum_{j=1}^n \varphi_j.
\end{align*}
We interpret the step function $\varphi$ as a particle density function constructed from $\hat \bx$. By construction, it satisfies $\supp \varphi = \supp \orho$ and $\int_{I_i} \varphi = \frac1n$ for all $1 \leq i \leq n$. By Proposition \ref{prop:orhoa:sup:bd} we obtain that
\begin{equation} \label{pf:zj}
  \ell_i 
  \geq \frac1{n \| \orho \|_\infty }  
  \geq \frac c{q_\alpha n}
  \quad \text{for all } 1 \leq i \leq n
\end{equation}
for some $c > 0$ independent of $i,n,\alpha$.

Next we define the `diagonal contribution' to the energies $\hat E_n^\alpha $ and $E^\alpha$. Let
\[
  D_n(\hat \bx) 
  := \frac1{n^2} \sum_{i =1}^n K_\alpha (\hat x_i - \hat x_{i-1})
  = \frac1{n^2} \sum_{i =1}^n K_\alpha (\ell_i)
\]
be the contribution of the neighboring particles to the interaction part of the energy $\hat E_n^\alpha (\hat \bx)$, and let
\begin{equation} \label{pf:zu} 
  D(\varphi)
  := \frac12 \sum_{i=1}^n \int_\R (K_\alpha * \varphi_i) \varphi_i
  = \frac1{2 n^2}  \sum_{i=1}^n \frac1{\ell_i^2} \iint_{(0,\ell_i)^2} K_\alpha (x-y) \, dydx.
\end{equation}
be a certain diagonal part of $E^\alpha (\varphi)$. Since $K$ is decreasing on $(0,\infty)$, we observe that 
\[
  D_n(\hat \bx) \leq 2 D(\varphi).
\]

Next we show that
\begin{equation} \label{pf:zt}
  D(\varphi) \leq C \frac \alpha n \log \frac{ q_\alpha n }\alpha.
\end{equation}
We start with estimating the integral in the right-hand side of \eqref{pf:zu}. Rotating the integration domain $(0,\ell_i)^2$ by $45$ degrees, and covering the resulting diamond-shaped domain with a square of size $\sqrt 2 \ell_i$, we estimate
\begin{equation*}
  \iint_{(0,\ell_i)^2} K_\alpha (x-y) \, dydx
  \leq 2 \sqrt 2 \ell_i \int_0^{\ell_i/\sqrt 2} K_\alpha (\xi) \, d\xi
  = 2 \sqrt 2 \ell_i \int_0^{\alpha \ell_i/\sqrt 2} K (\xi) \, d\xi.
\end{equation*}
Then, using the estimate \eqref{pf:zl} 
we obtain for $\ell_i \leq 1/(\sqrt 2 \alpha)$ that
\[
  \int_0^{\alpha \ell_i/\sqrt 2} K (\xi) \, d\xi
  \leq C_1 \int_0^{\alpha \ell_i/\sqrt 2} - \log ( \xi) \, d\xi
  \leq C \alpha \ell_i \log \frac1{ \alpha \ell_i }
\]
and for $\ell_i > 1/(\sqrt 2 \alpha)$ that
\[
  \int_0^{\alpha \ell_i/\sqrt 2} K (\xi) \, d\xi
  \leq \int_0^{\infty} K (\xi) \, d\xi
  \leq C < C \sqrt 2 \alpha \ell_i.
\]
Applying these estimates to \eqref{pf:zu}, we obtain
\[
  D(\varphi)
  \leq \frac C{n^2}  \sum_{i=1}^n \alpha \log \Big( \max \Big\{ \frac1{ \alpha \ell_i }, e \Big\} \Big).
\]
Recalling \eqref{pf:zj} and using that $\alpha \leq n/2$ and $q_\alpha \geq 1$,  \eqref{pf:zt} follows. \comm{This is a bit annoying. The only interest is the upper bound on $1/(\alpha \ell_i)$. However, since it is in the logarithm, we technically also have to make sure that it is bounded from below by $2$ }

With these preparations in hand we prove \eqref{pf:zs}. We start with the confinement part, and show that 
\begin{equation} \label{pf:zq}
  \frac1n \sum_{i=0}^n Q_\alpha (  \hat x_i ) - \sum_{i=1}^n \int_{I_i} Q_\alpha \, d\orho 
  \leq \frac Cn.
\end{equation}
By Lemma \ref{l:orhoa:supp}, $0 \in [y_1, y_2]$, and thus there exists a $j$ such that $0 \in I_j$. Recalling \eqref{pf:zr} and the monotonicity of $Q$ on either side of $0$, we estimate
\[
  \int_{I_i} Q_\alpha \, d\orho
  \geq \frac1n \left\{ \begin{array}{ll}
    Q_\alpha(\hat x_i)
    & \text{if } i < j \\
    0
    & \text{if } i = j \\
    Q_\alpha(\hat x_{i-1})
    & \text{if } i > j. \\
  \end{array} \right.
\]
Then,
\[
  \frac1n \sum_{i=0}^n Q_\alpha (  \hat x_i ) - \sum_{i=1}^n \int_{I_i} Q_\alpha \, d\orho 
  \leq \frac{Q_\alpha (  \hat x_0 ) + Q_\alpha (  \hat x_n )}n
  = \frac{Q_\alpha ( y_1 ) + Q_\alpha ( y_2 )}n,
\]
and \eqref{pf:zq} follows from \eqref{l:orhoa:supp:Qn:UB}.

It is left to prove \eqref{pf:zs} for the interaction part. With this aim, we set $\hEnint$ and $\Eint$ as the interaction parts of respectively $E_n^\alpha$ and $E^\alpha$. Then, the proofs of \cite[Lemma 4.3]{KimuraVanMeurs21} and \cite[Lemma 4.5]{KimuraVanMeurs21} demonstrate that, respectively,
\begin{align*}
  \Eint(\varphi) - \Eint(\orho) &\leq 3 D(\varphi) + 4 D_n(\hat \bx) \quad \text{and} \\
  \hEnint (\hat \bx) - \Eint(\varphi) &\leq D_n(\hat \bx)  
\end{align*}
hold.
Recalling $D_n(\hat \bx) \leq 2 D(\varphi)$ and \eqref{pf:zt}, the desired estimate in \eqref{pf:zs} follows by adding the two estimates in the display above. 

\subsection{Lower bound on the energy difference}
\label{s:pf:t:E:LB}

In this section we prove the following of the four inequalities in Theorem \ref{t}:
\begin{align} \label{pf:zo}
  E_n^\alpha (\obx) - E^\alpha (\orho) 
   \geq -C \frac{\alpha}n \log \frac{ q_\alpha n }\alpha,
\end{align}
where $q_\alpha$ is defined in \eqref{qal}.

The proof below is inspired by the proof of \cite[Proposition 5.1]{KimuraVanMeurs21}. 
Related to $\obx$ and $\orho$, we define the measures
\begin{equation*} %\label{mun}
  \mu_n := \frac1n \sum_{i=1}^n \delta_{\ox_i}
  \quad \text{and} \quad
  \nu_n := \mu_n - \orho.
\end{equation*}
Let
$
  \Delta := \{ (x, x) : x \in \R \} \subset \R^2
$
be the diagonal in $\R^2$, and note that 
\begin{equation*} %\label{En:mun}
  E_n^\alpha(\obx) 
  = \frac12 \iint_{\Delta^c} K_\alpha(x-y) \, d(\mu_n \otimes \mu_n)(x,y) + \int_\R Q_\alpha \, d\mu_n.
\end{equation*}
Then, 
\begin{align} \label{pf:zy}
  E_n^\alpha(\obx) - E^\alpha ( \orho ) = \frac12 \iint_{\Delta^c} K_\alpha (x-y) \, d(\nu_n \otimes \nu_n)(x,y)
    + \int_\R (K_\alpha * \orho) d\nu_n
    + \int_\R Q_\alpha \, d\nu_n.
\end{align}
From the Euler-Lagrange equation \eqref{EL:ineq:on:R} it follows that the contribution of the last two integrals is nonnegative.  

Next, let $L$ and $M$ be as defined in \eqref{KLM} for some $\sigma \in (0,\frac12]$ which we choose later. Analogously to $K_\alpha$ in \eqref{Kal}, we define $L_\alpha$ and $M_\alpha$. Then, we obtain for the first integral in the right-hand side of \eqref{pf:zy} that
\begin{multline*}
   \iint_{\Delta^c} K_\alpha (x-y) \, d(\nu_n \otimes \nu_n)(x,y) \\
   = \iint_{\Delta^c} M_\alpha (x-y) \, d(\nu_n \otimes \nu_n)(x,y)
     + \iint_{\R^2} L_\alpha (x-y) \, d(\nu_n \otimes \nu_n)(x,y)
     - \frac1n L_\alpha(0).
\end{multline*} 
For the first integral, we expand $\nu_n = \mu_n - \orho$ and obtain
\begin{multline*}
  \iint_{\Delta^c} M_\alpha (x-y) \, d(\nu_n \otimes \nu_n)(x,y) \\
  = \frac1{n^2} \sum_{i \neq j} M_\alpha (\ox_i - \ox_j)
    - 2 \int_\R (M_\alpha * \orho) \, d \mu_n
    + \iint_{\R^2} M_\alpha (x-y) \orho (y) \orho(x) \, dydx.
\end{multline*}
Since $M_\alpha \geq 0$, the first and the third integral are nonnegative. For the second integral, we apply Proposition \ref{prop:orhoa:sup:bd} and Lemma \ref{l:KLM}\ref{l:KLM:Mint} to obtain
\begin{equation*} 
  \bigg| \int_\R (M_\alpha * \orho) \, d \mu_n \bigg| 
  \leq \int_\R \Big( \int_\R M_\alpha \Big) \| \orho \|_\infty \, d\mu_n
  \leq C  q_\alpha \sigma \log \frac1\sigma.
\end{equation*} 
Collecting all estimates above, we observe that
\begin{equation} \label{pf:zz}
  E_n^\alpha(\bx) - E^\alpha ( \orho ) 
  \geq \frac12 \int_\R (L_\alpha * \nu_n) \, d\nu_n - C  q_\alpha \sigma \log \frac1\sigma - \frac1n L_\alpha(0).
\end{equation}
By Lemma \ref{l:KLM}\ref{l:KLM:Lnorm}, the first term is nonnegative. 
Applying Lemma \ref{l:KLM}\ref{l:KLM:L0} to the third term, we get
\[
  \frac1n L_\alpha(0)
  \leq C \frac \alpha n \log \frac1\sigma.
\]
Then, by taking 
\[
  \sigma = \frac\alpha{q_\alpha n},
\] 
we obtain \eqref{pf:zo} from \eqref{pf:zz} (we have assumed here that $\alpha \leq \frac n2$ such that $\sigma \leq \frac12$).

\subsection{Lower bound on the potential difference}
\label{s:pf:t:F:LB}

In this section we prove the last of the four inequalities in Theorem \ref{t}:
\begin{align} \label{pf:zn}
  F_n^\alpha - F^\alpha 
   \geq -C \sqrt{ \frac{\alpha}n \log ( q_\alpha n) },
\end{align}
where $q_\alpha$ is defined in \eqref{qal}.

Similar to \eqref{pf:zy} we expand
\begin{multline} \label{pf:zv} 
  F_n^\alpha - F^\alpha = \frac12 \iint_{\Delta^c} K_\alpha (x-y) \, d\nu_n(y) d\nu_n(x) 
    + \frac12 \bigg( \int_\R (K_\alpha * \orho) d\nu_n
    + \int_\R Q_\alpha \, d\nu_n \bigg) \\
    + \frac12 \int_\R (K_\alpha * \orho) d\nu_n.
\end{multline}
Applying the argument which follows \eqref{pf:zy}, we bound the first two terms from below by $- C' \frac\alpha n \log \frac{ q_\alpha n }\alpha$. We expand the integral in the third term as
\begin{align} \label{pf:zx} 
  \int_\R (K_\alpha * \orho) d\nu_n
  = \int_\R (L_\alpha * \orho) d\nu_n
    + \int_\R (M_\alpha * \orho) d\nu_n,
\end{align}
where we take again $\sigma = \alpha / (q_\alpha n)$ in the definition of $L$ and $M$.
We bound the second term in the right-hand side of \eqref{pf:zx} as
\[
  \bigg| \int_\R (M_\alpha * \orho) d\nu_n \bigg| 
  \leq \| \orho \|_\infty \bigg( \int_\R M_\alpha \bigg) \int_\R d |\nu_n|
  \leq C \frac{\alpha}n \log ( q_\alpha n). 
\] 
For the first term in the right-hand side of \eqref{pf:zx}, we apply Lemma \ref{l:KLM}\ref{l:KLM:Lnorm} to obtain
\[
  \bigg|  \int_\R (L_\alpha * \orho) d\nu_n \bigg| 
  = \big| (\orho, \nu_n)_{L_\alpha} \big|
  \leq \| \orho \|_{L_\alpha} \| \nu_n \|_{L_\alpha}.
\]
To bound $\| \nu_n \|_{L_\alpha}$, we recall \eqref{pf:zz} and apply \eqref{pf:zp} to obtain
\begin{align*} %\label{pf:zw}
  \frac12 \| \nu_n \|_{L_\alpha}^2
  = \frac12 \int_\R (L_\alpha * \nu_n) \, d\nu_n
  \leq E_n^\alpha(\obx) - E^\alpha (\orho) + C \frac{\alpha }n \log \frac{ q_\alpha n }\alpha
  \leq C' \frac{\alpha}n \log \frac{ q_\alpha n }\alpha.
\end{align*}
In conclusion
\[
  \bigg|  \int_\R (L_\alpha * \orho) d\nu_n \bigg| 
  \leq \| \orho \|_{L_\alpha} \| \nu_n \|_{L_\alpha}
  \leq C' \sqrt{ \frac{\alpha}n \log \frac{ q_\alpha n }\alpha } \| \orho \|_{K_\alpha}.
\]
Since $\| \orho \|_{K_\alpha} \leq \sqrt{2 E^\alpha(\orho) }$ is uniformly bounded (see Lemma \ref{l:Ea:bds}), it can be absorbed in the constant $C'$. Finally, applying the obtained estimates to \eqref{pf:zv}, the desired estimate \eqref{pf:zn} follows. This completes the proof of Theorem \ref{t}.

%---------------------------------
\section{Proof of Theorem \ref{t:intro}}
\label{s:pf:t:intro} 

In this section we prove Theorem \ref{t:intro} by showing that it is a corollary of Theorem \ref{t}. Given $n, \beta, \alpha$ as in Theorem \ref{t:intro}, it follows from $\frac n{P(n)} \leq \beta \leq \Gamma n$ that
\[
  \alpha = P^{-1} \Big( \frac n\beta \Big) \in ( c, n], \qquad c := \max \{q_2, P^{-1}(1/\Gamma)\}.
\]
Hence, Theorem \ref{t} applies. From the definitions \eqref{E:def} and \eqref{Ena:scaling} with $\gamma$ as in \eqref{alpha} we obtain
\begin{align*}  
  I_n^D(\oba) - I_n^C(\omu) &= 2\frac{n^2}\alpha \big( E_n^\alpha(\obx) - E^\alpha (\orho)  \big) \\
  F_n^D - F_n^C &= 2\frac{n^2}\alpha \big( F_n^\alpha - F^\alpha \big).
\end{align*}
Then, applying the estimates of Theorem \ref{t} to the right-hand sides, the estimates in Theorem \ref{t:intro} follow from some elementary algebra.

\appendix

\section{Proof of Lemma \ref{l:convo:pp}}
\label{a:pf:l:convo}

Here we prove Lemma \ref{l:convo:pp}.
By changing variables if needed we may assume that $(a,b) = (0,1)$ and $\alpha = 1$.

First, we show that the integral in the right-hand side is well defined. Since $f \in L^1(\R)$ and $K''$ is bounded outside of any neighbourhood around $0$, it suffices to prove integrability of the integrand in a neighbourhood around $0$. Note that $\varphi(z) := f(x) + z f'(x)$ is the tangent line of $z \mapsto f(x+z)$ at $z=0$. Hence, the part inside the brackets is bounded in absolute value by $z \mapsto C z^2$ in a neighbourhood of $0$. Since $|K''(z)| \leq C z^{-2}$, the integrand is bounded on this neighbourhood.

Next we prove Lemma \ref{l:convo:pp} under the additional assumption that $f \in C^2(\R)$ with $f, f'$ bounded. Let
\[
  F(z) := f(x + z) - f(x) - z f'(x).
\]
Then, starting from the right-hand side in Lemma \ref{l:convo:pp},
\begin{multline} \label{pf10}
  \int_{\R} F(z) K''(z) \, dz 
  \xleftarrow{\e \to 0} \int_{\R \setminus B(0,\e)} F(z) K''(z) \, dz \\
  = \big[ F(z) K'(z) \big]_{-\infty}^{-\e}
     + \big[ F(z) K'(z) \big]_{\e}^\infty
     + \int_{\R \setminus B(0,\e)} F'(z) K'(z) \, dz.
\end{multline}
Since $F(z)$ has linear growth and $K(z)$ decays exponentially as $z \to \infty$, the boundary terms equal
\begin{align*}
  \big[ F(z) K'(z) \big]_{-\infty}^{-\e}
     + \big[ F(z) K'(z) \big]_{\e}^\infty
  &= F(-\e) K'(-\e) - F(\e) K'(\e) \\
  &= - \e^2 K'(\e) \frac{F(-\e) + F(\e)}{\e^2} \\
  &= - \e^2 K'(\e) \frac{f(x-\e) + f(x+\e) - 2f(x)}{\e^2}.
\end{align*}
Since $f \in C^2(\R)$, this term vanishes in the limit $\e \to 0$. For the remaining integral in \eqref{pf10}, we compute
\begin{align} \label{pf11}
  \int_{\R \setminus B(0,\e)} F'(z) K'(z) \, dz
  = \big[ F'(z) K(z) \big]_{-\infty}^{-\e}
     + \big[ F'(z) K(z) \big]_{\e}^\infty
     + \int_{\R \setminus B(0,\e)} F''(z) K(z) \, dz.
\end{align}
Since $F'$ is bounded, the boundary terms can be treated as above. This yields
\begin{align*}
  \big[ F'(z) K(z) \big]_{-\infty}^{-\e}
     + \big[ F'(z) K(z) \big]_{\e}^\infty
  = F'(-\e) K(-\e) - F'(\e) K(\e)
  = \e K(\e) \frac{ f(x - \e) - f(x + \e) }\e,
\end{align*}
which vanishes in the limit $\e \to 0$. For the remaining integral in \eqref{pf11} we obtain
\begin{align*}
  \int_{\R \setminus B(0,\e)} F''(z) K(z) \, dz
  &= \int_{\R \setminus B(0,\e)} f''(x + z) K(z) \, dz \\
  &= \int_{\R \setminus B(0,\e)} f''(x - y) K(y) \, dz \\
  &= (K*f'')(x)
  = (K*f)''(x).
\end{align*}
This proves Lemma \ref{l:convo:pp} under the additional assumption that $f \in C^2(\R)$ with $f, f'$ bounded.

Next we prove Lemma \ref{l:convo:pp} in the general case. For convenience, we assume $x \leq \frac12$ and set $\delta := x/3$. \comm{We use this to work with the two concentric balls $B(x, \delta)$ and $B(x, 2\delta)$. $\tilde f$ is such that it is equal to $f$ on $B(x, 2\delta)$. For $\tilde f$ to be in $C^2$ and extendable to $\R$ in a $C^2$-manner, we need $B(x, 2\delta)$ to stay a positive distance away from $0$ and $1$. In addition, later, we will regularize $K$ on $B(0, \delta)$, and in the convolution in the last formula below we then need room to shift the ball of radius $\delta$ inside $B(x, 2\delta)$. In more detail, we use that $B(\xi, \delta) \subset B(x, 2\delta)$ for all $\xi \in B(x, \delta)$. } 
Let $\tilde f$ be a $C^2$ extension of $f |_{B(x,2\delta)}$ to $\R$ such that $\tilde f \in L^1(\R)$ with $\tilde f$ and $\tilde f'$ bounded. Then, by what we have just proven,
\[
  (K * \tilde f)''(x)
  = \int_{\R} \big[ \tilde f(x + z) - \tilde f(x) - z \tilde f'(x) \big] \, K''(z) dz,
\]
and it remains to prove Lemma \ref{l:convo:pp} for
\[
  g := f - \tilde f.
\]
Note that $g \in L^1(\R)$ satisfies 
\begin{equation} \label{pf13}
  B(x,2\delta) \cap \supp g = \emptyset.
\end{equation}
Then, for any even $C^2$-extension $\tilde K$ of $K |_{\R \setminus B(0, \delta)}$ to $\R$, the right-hand side of Lemma \ref{l:convo:pp} can be rewritten as
\begin{align*}
  \int_\R g(x+z) K''(z) \, dz
  = \int_\R g(x-y) \tilde K''(y) \, dz 
  = (\tilde K'' * g)(x)
  = (\tilde K * g)''(x).
\end{align*}
Finally, to show that $(\tilde K * g)''(x) = (K * g)''(x)$, we take any $\xi \in B(x, \delta)$, and observe from \eqref{pf13} that
\[
  (\tilde K * g)(\xi)
  = \int_{B(0,\delta)^c} g(\xi-y) \tilde K(y) \, dz
  = \int_{B(0,\delta)^c} g(\xi-y) K(y) \, dz
  = (K * g)(\xi).
\]
This completes the proof of Lemma \ref{l:convo:pp}.

\section*{Acknowledgements}

PvM was supported by JSPS KAKENHI Grant Number JP20K14358.

\end{document}